\documentclass{scrarticle}
\usepackage[utf8]{inputenc}
\usepackage{authblk}

\usepackage{graphicx}
\usepackage{amsthm}
\usepackage{amsmath}
\usepackage{amssymb}
\usepackage{float}
\usepackage{mathtools}
\usepackage{dsfont}
\usepackage{tikz}
\usetikzlibrary{decorations.pathreplacing,math}
\usepackage{hyperref}
\usepackage{cleveref}
\usepackage[textwidth=1.1cm,textsize=scriptsize,shadow]{todonotes}
\usepackage[labelfont=bf]{caption}
\usepackage{subcaption}

\let\phi\varphi
\let\epsilon\varepsilon

\let\vec\boldsymbol
\newcommand*	{\NN}	{\mathbb{N}}
\newcommand*	{\RR}	{\mathbb{R}}
\newcommand*	{\DD}	{\mathbb{D}}
\newcommand*	{\EE}	{\mathbb{E}}
\newcommand*    {\CT}   {\mathcal{T}}

\newcommand     {\deltagraph}   {$\delta$-neigh\-bor\-hood graph}
\newcommand     {\completegraph}    {fully connected graph}
\newcommand     {\basecase}     {base case}
\newcommand*    {\edgecon}[1]  {\kappa_{#1}}

\newcommand*    {\compl}[1]{{#1}^{\mathsf{c}}}
\newcommand*\mean[1]{\bar{#1}}
\DeclarePairedDelimiter{\size}{\lvert}{\rvert}
\DeclareMathOperator{\ord}{ord}
\DeclareMathOperator*{\Var}{Var}
\DeclareMathOperator{\Cov}{Cov}

\DeclareMathOperator*{\EX}{\mathbb{E}}

\newtheorem{theorem}{Theorem}
\newtheorem{lemma}[theorem]{Lemma}
\newtheorem{corollary}[theorem]{Corollary}
\newtheorem{notation}[theorem]{Notation}
\newtheorem{remark}[theorem]{Remark}
\newtheorem{example}[theorem]{Example}
\newtheorem{definition}[theorem]{Definition}

\def\todoheader#1{\color{black}\normalfont\sffamily\bgroup\makebox[2em]{\bfseries #1\hfill}\egroup}

\title{Untangling Gaussian Mixtures}
	\author[1]{Eva Fluck}
	\affil[1]{Department of Computer Science, RWTH Aachen University, \textit{lastname@cs.rwth-aachen.de}}
	
	\author[2]{Sandra Kiefer}
	\affil[2]{Department of Computer Science, University of Oxford, \textit{sandra.kiefer@cs.ox.ac.uk}}
	
	\author[1]{Christoph Standke}

\begin{document}
	\maketitle
	
	\begin{abstract}
		Tangles were originally introduced as a concept to formalize regions of high connectivity in graphs. In recent years, they have also been discovered as a link between structural graph theory and data science: when interpreting similarity in data sets as connectivity between points, finding clusters in the data essentially amounts to finding tangles in the underlying graphs.
		
		This paper further explores the potential of tangles in data sets as a means for a formal study of clusters. Real-world data often follow a normal distribution. Accounting for this, we develop a quantitative theory of tangles in data sets drawn from Gaussian mixtures. To this end, we equip the data with a graph structure that models similarity between the points and allows us to apply tangle theory to the data.
		We provide explicit conditions under which tangles associated with the marginal Gaussian distributions exist asymptotically almost surely. This can be considered as a sufficient formal criterion for the separabability of clusters in the data.
	\end{abstract}

\section{Introduction}

Cluster analysis \cite{cohen19, addad2017objective, dasgupta16hcoptimization, hennig2015handbook} can be summarized as the task of grouping data points based on some similarity measurement.
There are two central paradigms for clustering algorithms:
some, like $k$-means and $k$-median, focus on \emph{similarity} and group together data points that are sufficiently close to each other, whereas for example the spectral clustering approach rather uses \emph{dissimilarity} to split the data into sufficiently different groups.
Either way, hard clustering, i.e.\ clustering that partitions the data into disjoint subsets, has its limits when it comes to real-world data. These data sets typically have some inherent fuzziness and clusters that are most adequate to represent them may be wide-spread and might even overlap. All in all, in real-world data, we cannot always expect a clear answer to the question whether two points should be grouped together or considered different, and soft-clustering approaches are necessary. The most common such approach is to assign values to the data points that denote the degree of membership to each cluster. For an overview, see e.g.\ \cite{de2007advances}.

Connectivity is a central concept for studying the structure of complex graphs and more general relational structures \cite{diestel2016,Diestel_Oum_2014,geelen,minorsX}. It finds its way into cluster analysis when we consider ``similar'' data points as adjacent or connected (see e.g.\ \cite{luxburg2007spectral}).  Here, the similarity is usually measured via the distance of the data points in~$\mathbb{R}^d$. When we equip a set of data points with a suitable graph structure that reflects their proximity, analyzing similarity of data points becomes a problem of studying the graph connectivity. We can hence borrow also graph-theoretic concepts to analyse the data set. As it turns out, the notion of a graph \emph{tangle} captures exactly the kind of ambiguity described above: it is a precise formal definition of potentially fuzzy but sufficiently distinct clusters in graphs. 

Tangles in graphs were first introduced in \cite{minorsX}. They offer a perspective on connectivity that differs from the classical view: instead of focusing on the highly connected regions, they deal with thin parts, i.e., the bottlenecks in the graph. A highly connected region is never located exactly at a bottleneck; it must lie (mostly) on one of the sides defined by a cut along the bottleneck.
A tangle is a set of consistent orientations of all the bottlenecks, where the orientation is towards the side of the bottleneck on which the major part of the associated highly connected region lies.
So a tangle provides ``signposts'' towards a particular highly connected region and assigns one such signpost to each low-order cut.
We identify highly connected regions by looking at the signposts. That is, we consider the regions different whenever there is at least one signpost that points into opposite directions in the two tangles. This notion of being sufficiently different is somewhat along the lines of the dissimilarity paradigm, but in a less strict sense than clusters. In fact, the definition via consistent orientations of bottlenecks is robust precisely because it avoids references to an explicit point set.

The authors of \cite{minorsX} show a powerful duality result, which links tangles closely to branch decompositions of graphs. Furthermore, they give a tree decomposition of a graph into its maximal tangles, which was turned into a canonical one in \cite{Carmesin2016}.
The original tangle concept was extended to directed graphs \cite{Johnson2001, Reed1999} and tangle-tree decompositions were also found in that setting \cite{kreuzer2020}.
In related work, variations of the tangle concept that consider different ways to measure connectivity in graphs \cite{adler2007,bozdefokpil21} and in matroids \cite{geelen} have been explored, see also \cite{grohe2016}.
The insight that one can define tangles completely without an underlying graph structure was carried further and led to the study of \emph{abstract separation systems} \cite{dies18,dieserweiss19,diesoum21,elbkneitee21}, where separations play the role of the previously mentioned cuts.
Even in this very general setting, duality theorems which relate tangles to tree decompositions have been derived \cite{diesoum21}.

With Data Science and Machine Learning emerging as central research fields in modern computer science, tangles have caught attention also in the context of clustering.
The first transfer of the concept to real-world data happened in digital-image analysis \cite{diestel2016tangles}, where tangles were used to describe clusters in image data in order to find interesting image parts. As Diestel puts it in \cite{dies20}, ``tangles offer a new paradigm for clustering in large data sets''. Informally speaking, unlike clusters, tangles come equipped with an \emph{order}, a narrowness threshold that defines how small a cut through the data must be for the corresponding region to be considered a bottleneck.
This means that a parameter of the tangle is the ``width of the bottlenecks'' at the signposts that constitute the tangle, and this way, the order is an indicator for the unambiguity of the highly-connected region or the cluster associated with the tangle. Vice versa, branch decompositions serve as witnesses for the absence of tangles above a certain order and may thereby be considered indicators that clusters above a certain quality threshold do not exist within the data. Additionally, the order of the tangles captures hierarchical structure of the data: by decreasing the threshold, some of the previously considered cuts are not classified as small anymore, so the new bottlenecks are a subset of the previous ones and thus, several of the previously distinct dense regions may be united into a single one. It is worth mentioning here that also the tangle-tree decomposition has its correspondence in the realm of clustering, as it yields a canonical decomposition of the data set into clusters, see \cite{diestel2016}.

Tangles can also help when it comes to interpreting clusters. 
In some scenarios, separations of the data have semantics (for example, if one separates discrete data by type) and the orientations of those can then give an explanation of the similarities between clustered data points. For more background on interpretable classification and interpretable clustering, we refer the reader to \cite{BreimanFOS84,PellegM01}.

So far, the connection between tangles and clustering is largely based on intuition, as the numerous figurative expressions used in this introduction indicate.
Towards grasping the link formally, \cite{fluck19} establishes an explicit connection between tangles and hierarchical clustering, which can be generalized to a one-to-one correspondence between a certain type of connectivity functions together with its tangles and hierarchical clustering \cite{fluck22}. Furthermore, the authors of~\cite{elbfioklekneirendteelux20} introduce an algorithmic framework to compute a clustering based on tangles.

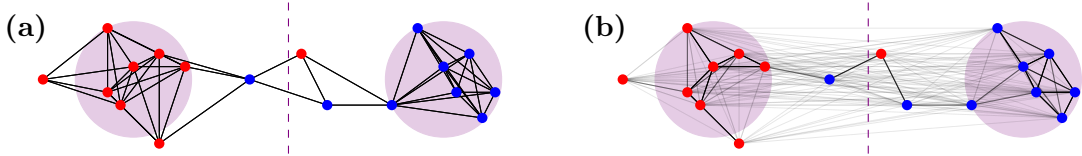
\begin{figure}
        \begin{subfigure}[t]{0.49\textwidth}
                                \begin{tikzpicture}
                                \node[anchor=north east] at (-1.5, 1) {};
                                \node[anchor=north east] at (-1, 1) {\llap{\textbf{(a)}}};
                                \begin{scope}[scale=0.34]
                                \filldraw[violet, opacity=0.2] (0,0) circle (2.25);
                                \filldraw[violet, opacity=0.2] (12,0) circle (2.25);

                                \node (a) at (0,0.5) {};
                                \node (b) at (-1,-0.5) {};
                                \node (c) at (1,-2.5) {};
                                \node (d) at (2,0.5) {};
                                \node (e) at (-0.5,-1) {};
                                \node (f) at (-3.5,0) {};
                                \node (g) at (1,1) {};
                                \node (h) at (-1,2) {};
                                \node (i) at (6.5,1) {};

                                \node (l) at (4.5,0) {};
                                \node (m) at (10,-1) {};
                                \node (n) at (11,2) {};
                                \node (o) at (13,1) {};
                                \node (p) at (14,-0.5) {};
                                \node (q) at (12.5,-0.5) {};
                                \node (r) at (13.5,-1.5) {};
                                \node (s) at (7.5,-1) {};
                                \node (t) at (12,0.5) {};

                                \foreach \a/\b in {0/0.5, -1/-0.5, 1/-2.5, 2/0.5, -0.5/-1, -3.5/0, 1/1, -1/2, 6.5/1, 4.5/0, 10/-1, 11/2, 13/1, 14/-0.5, 12.5/-0.5, 13.5/-1.5, 7.5/-1, 12/0.5}{
                                        \foreach \c/\d in {0/0.5, -1/-0.5, 1/-2.5, 2/0.5, -0.5/-1, -3.5/0, 1/1, -1/2, 6.5/1, 4.5/0, 10/-1, 11/2, 13/1, 14/-0.5, 12.5/-0.5, 13.5/-1.5, 7.5/-1, 12/0.5}{
                                                \tikzmath{
                                                        \x = \a - \c;
                                                        \y = \b - \d;
                                                        \dist = sqrt((\x)^2+(\y)^2);
                                                        if \dist <= 4.5 then {
                                                                {\draw (\a,\b) -- (\c,\d);};
                                                        };
                                                }
                                        }
                                }
                                \foreach \a in {a,b,c,d,e,f,g,h,i}{
                                        \fill[red] (\a) circle (0.2);
                                }
                                \foreach \a in {l,m,n,o,p,q,r,s,t}{
                                        \fill[blue] (\a) circle (0.2);
                                }

                                \draw[color=violet, dashed] (6,3) -- (6,-3);
                                \end{scope}
                                \end{tikzpicture}
                \phantomcaption{~}
                \label{fig:datamodel_intuition2}
        \end{subfigure}~
        \hspace{0.7cm}\begin{subfigure}[t]{0.49\textwidth}
                        \begin{tikzpicture}
                        \node[anchor=north east] at (-1, 1) {\llap{\textbf{(b)}}};
                        \begin{scope}[scale=0.34]
                        \filldraw[violet, opacity=0.2] (0,0) circle (2.25);
                        \filldraw[violet, opacity=0.2] (12,0) circle (2.25);

                        \node (a) at (0,0.5) {};
                        \node (b) at (-1,-0.5) {};
                        \node (c) at (1,-2.5) {};
                        \node (d) at (2,0.5) {};
                        \node (e) at (-0.5,-1) {};
                        \node (f) at (-3.5,0) {};
                        \node (g) at (1,1) {};
                        \node (h) at (-1,2) {};
                        \node (i) at (6.5,1) {};

                        \node (l) at (4.5,0) {};
                        \node (m) at (10,-1) {};
                        \node (n) at (11,2) {};
                        \node (o) at (13,1) {};
                        \node (p) at (14,-0.5) {};
                        \node (q) at (12.5,-0.5) {};
                        \node (r) at (13.5,-1.5) {};
                        \node (s) at (7.5,-1) {};
                        \node (t) at (12,0.5) {};

                        \foreach \a/\b in {0/0.5, -1/-0.5, 1/-2.5, 2/0.5, -0.5/-1, -3.5/0, 1/1, -1/2, 6.5/1, 4.5/0, 10/-1, 11/2, 13/1, 14/-0.5, 12.5/-0.5, 13.5/-1.5, 7.5/-1, 12/0.5}{
                                \foreach \c/\d in {0/0.5, -1/-0.5, 1/-2.5, 2/0.5, -0.5/-1, -3.5/0, 1/1, -1/2, 6.5/1, 4.5/0, 10/-1, 11/2, 13/1, 14/-0.5, 12.5/-0.5, 13.5/-1.5, 7.5/-1, 12/0.5}{
                                        \tikzmath{
                                                \x = \a - \c;
                                                \y = \b - \d;
                                                \dist = (\x)^2+(\y)^2;
                                                \w = exp(-(\dist)/2)*7 + 0.05;
                                                {\draw[opacity=\w] (\a,\b) -- (\c,\d);};
                                        }
                                }
                        }
                        \foreach \a in {a,b,c,d,e,f,g,h,i}{
                                \fill[red] (\a) circle (0.2);
                        }
                        \foreach \a in {l,m,n,o,p,q,r,s,t}{
                                \fill[blue] (\a) circle (0.2);
                        }
                        \draw[color=violet, dashed] (6,3) -- (6,-3);
                        \end{scope}
                        \end{tikzpicture}
                \phantomcaption{~}
                \label{fig:datamodel_intuition3}
        \end{subfigure}
        \caption{Data drawn from two Gaussians with their hidden labels in red and blue. The dashed line represents a possible low-order cut and the purple circles are candidates for highly connected regions. \textbf{(a)} A \deltagraph\ on the data, that is, two data points are adjacent if their distance is at most $\delta$. \textbf{(b)} The \completegraph\ with the edge weights dependent on the distance of the data points represented in the opacity.}
        \label{fig:datamodel_intuition}
\end{figure}
\paragraph*{Our contribution}
This paper embarks on a project to make the connection between tangles and clusters more robust, formal, and quantitative, without depending on particular clustering algorithms or paradigms. 

Real-world data frequently follow a normal distribution, and clustering algorithms are often tested against data drawn from Gaussian mixtures \cite{chauddasvat09,luxburg2007spectral}.
We develop a quantitative theory of tangles in data sets obtained from sampling from Gaussian mixtures, i.e., we assume that the data points are drawn from various Gaussian distributions. Our goal is to identify incomparable tangles, that is, pairs of tangles that orient at least one separation (cut) of the underlying graph differently or rather where none is just a restriction of the other. Our theory allows us to draw conclusions of the kind ``With high probability, there exist incomparable tangles in the data'' as well as concrete lower bounds on these probabilities. 
Each tangle will correspond to a marginal distribution of the Gaussian mixture, i.e., to one Gaussian distribution contributing to the mixture, so we can interpret a high probability for the existence of incomparable tangles as a formalization of the distinguishability or separability of the marginals. 

In a set of data points drawn from a Gaussian mixture, a reasonable clustering algorithm would typically determine clusters that closely correspond to the hidden labels, i.e., it will assign most points correctly to the marginal distribution they have been drawn from \cite{GiraudV18}.
However, as described above, the goal of this work is not to split the data up into partition classes or to actually determine hidden labels.
We rather study ways to quantify the existence of clusters (or tangles) in the data set, which could in a next step be interpreted as a quality measure for the reliability of a prediction of the hidden labels. 
To this end, we take up two standard ways to represent the data as a weighted graph, the \deltagraph\ (\Cref{fig:datamodel_intuition2}) and the \completegraph\ (\Cref{fig:datamodel_intuition3}). In both graphs, similarity of data points is represented in edges and weights on those. We thus use these edge weights as our connectivity measure for candidate bottlenecks to define the tangles.

Consider the following question: given a mixture of Gaussian distributions, what is the probability that incomparable tangles associated with the marginals exist when $n$ data points are drawn?
By carefully employing classical tools from probability theory, our analysis provides explicit conditions such that, asymptotically almost surely, the tangles exist.
In the \deltagraph, we further improve our lower bound on the probability: with (explicit) stronger assumptions, we deduce that the probability for the existence of incomparable tangles tends to $1$ exponentially fast in the number of data points.

Our theorems are applicable to any mixture of Gaussian distributions.
But in order to be so versatile, their preconditions are quite technical and it is not trivial to see what choices of parameters and sets fulfill them.
We analyze the preconditions of our theorems for specific sets of parameters to showcase what a ``good choice'' may be as well as to compute the explicit probability bounds in different settings. 

We perform a detailed analysis for the mixture of two $1$-dimensional Gaussians. Here (see Figure \ref{fig:small_n_1dim_delta_mu} -- \ref{fig:1dim_2distr_expectation_mixing_sigma}), we find that the means of the marginal distributions need not be far apart in order to satisfy the conditions for the applicability of our guarantees. Furthermore, the obtained lower probability bounds for the existence of incomparable tangles are close to $1$ even for small data sets. This excludes large hidden constants and shows that our bounds are not only asymptotically high. For the general situation, as soon as the dimension is larger than $1$ and we have at least three marginal Gaussians, the means of the distributions are not necessarily aligned. Accounting for this, we study examples of Gaussian mixtures whose marginals cannot be separated well by a hyperplane. 

Our work differs from \cite{elbfioklekneirendteelux20}, which also studies tangles in mixtures of two Gaussians, in that we consider all possible separations of the data set instead of restricting them. By studying specific well-behaved separations and imposing restrictions on the data sets, the authors of \cite{elbfioklekneirendteelux20} manage to provide theoretical guarantees that the tangles are in 1-to-1 correspondence with the marginal distributions of the mixture. Moreover, while the results in \cite{elbfioklekneirendteelux20} treat the special case of two Gaussians, our theorems may appear intimidating and therefore cumbersome to handle at first sight, but they apply to the general setting of any mixture of Gaussian distributions with arbitrary parameters.

\paragraph*{Outline} 
Our paper is structured as follows. \Cref{sec:prelim} of this paper provides an introduction to tangles on weighted graphs and some useful tools from probability theory. 
\Cref{sec:datamodel} is dedicated to our model assumptions and the two graph structures we assign to the data sets. In Section \ref{sec:delta}, we present our results for the \deltagraph\ and outline in Section \ref{sec:complete} to which extent they translate to the \completegraph. 
We conclude with a discussion of possible follow-up projects in Section \ref{sec:conclusion}.
\section{Preliminaries}\label{sec:prelim}

For $n\in\NN_{>0}$, we let $[n] \coloneqq \{1,\ldots,n\}$. For a finite set $U$, we write $\binom{U}{n}$ to denote all $n$-element subsets of $U$ and $2^U$ to denote the set of all subsets. For $S \subseteq U$, we let $\compl{S} \coloneqq U \setminus S$. 

\subsection{Graphs and tangles}

This subsection is supposed to serve the reader as a light introduction to the realm of connectivity functions and tangles. For more background on the topic, we refer to \cite{grohe2016}.

A \emph{weighted graph} is a triple $G=(V,E,w)$ where the pair $(V,E)$ is an undirected graph with vertex set $V$ and edge set $E \subseteq \binom{V}{2}$, and $w \colon E \to \RR_{>0}$ is a function that assigns \emph{weights} to the edges.
For $u,v \in V$ with $\{u,v\} \in E$, we let $w(u,v) \coloneqq w(\{u,v\})$. In the remainder of this paper, all graphs are assumed to be weighted. If $E = \binom{V}{2}$, the graph $G$ is a \emph{clique}. 
For $W \subseteq V$, we denote by $G[W]$ the (weighted) subgraph of $G$ induced by the vertices in $W$.
For $S,T\subseteq V$, we let $E(S,T) \coloneqq \{\{u,v\}\in E\mid u\in S,v \in T \}$. Furthermore, we set $E(S) \coloneqq E(S,S)$.

For a finite set $U$, we call $\kappa\colon 2^U\rightarrow \mathbb{R}$ a \emph{connectivity function} if all of the following hold.
\begin{itemize}
    \item $\kappa$ is \emph{symmetric}, i.e., for all $S\subseteq U$, it holds that $\kappa(S)=\kappa(\compl{S})$.
    \item $\kappa$ is \emph{normalized}, i.e., $\kappa(\emptyset)=\kappa(U)=0$.
    \item $\kappa$ is \emph{submodular}, i.e., $\kappa(S)+\kappa(T)\geq\kappa(S\cap T)+\kappa(S\cup T)$ holds for all $S,T\subseteq U$.
\end{itemize}

A prominent example for a connectivity function is the edge connectivity on graphs. Intuitively, it measures how strongly a vertex set is connected to its complement.
On a graph where all edges have weight $1$, the edge connectivity of a vertex set equals the order of a minimal edge separation, i.e., the minimum number of edges that need to be removed to separate the vertex set and its complement.

\begin{example}[see \cite{grohe2016}]
    Let $G=(V,E,w)$ be a weighted graph. For $S \subseteq V$, we define the \emph{edge-connectivity function} $\kappa_G\colon 2^V\rightarrow\mathbb{R}$ via
	\[
	\kappa_G(S) \coloneqq \sum_{\{u,v\}\in E(S,\compl{S})} w(u,v).\]
	\label{ex:edge-con}
\end{example}

Using connectivity functions, we can describe highly connected regions in structures formally. In the following, to do so, we introduce tangles based on connectivity functions. Every tangle corresponds to one highly connected region and vice versa. To have a robust notion of high connectedness, instead of describing the region pointwise, the tangle just provides directions towards the region in an unambiguous way. More precisely, for every set with a low value of the connectivity function (i.e., a ``low connectivity'' to its complement), the tangle contains either the set or its complement, which can be interpreted as determining on which side of the induced separation between the set and its complement the main part of the considered highly connected region lies.

\begin{definition}
    Let $U$ be a finite set and $\kappa$ a connectivity function on $U$. A \emph{$\kappa$-tangle} of order $k\geq0$ is a set $\mathcal{T}\subseteq 2^U$ such that all of the following hold.\footnote{This definition differs slightly from the early works on tangles. We orient all separations within the tangle towards the ``large side'', whereas the original definition in \cite{minorsX} orients separations towards the ``small side''. In the clustering setting, our orientation is more natural. Still, both definitions of the orientation yield equivalent theories.}
	\begin{description}
		\item[\textbf{T.0}] $\kappa(S)<k$ for all $S\in \mathcal{T}$.\label{it:bound:tangle}
		\item[\textbf{T.1}] For all $S\subseteq U$ with $\kappa(S)<k$, either $S\in \mathcal{T}$ or $\compl{S}\in \mathcal{T}$ holds.\label{it:xor:tangle}
		\item[\textbf{T.2}] For all $S_1,S_2,S_3\in \mathcal{T}$, it holds that $S_1\cap S_2\cap S_3\neq\emptyset$.
		\item[\textbf{T.3}] For all $x\in U$, it holds that $\{x\}\notin\mathcal{T}$.
	\end{description}
	If $\kappa$ is clear from the context, we drop it and we call $\mathcal{T}$ a tangle. We denote the order of $\mathcal{T}$ by $\ord(\mathcal{T})$.
	\label{def:tangle}
\end{definition}

Conditions (T.0) and (T.1) ensure that a tangle orients all and only the separations of low order.
The intuition behind requiring non-empty intersection in condition (T.2) is to ensure that the orientations induced by the tangle are consistent, that is, all point towards a unique highly connected region. The number $3$ might seem arbitrary here, but it turns out that it is necessary and sufficient to give a rich theory (see also \cite{Diestel_Oum_2014,grohe2016}). Finally, condition (T.3) means that singletons are not considered to be highly connected regions.

We call two tangles $\mathcal{T}_1$, $\mathcal{T}_2$ \emph{incomparable} if there is an $S\subset U$ with $S\in \CT_1$ and $\compl{S}\in \CT_2$.
That is, the tangles are incomparable if there is some separation that is oriented differently by them.
If $\mathcal{T}_1$ and $\mathcal{T}_2$ are not incomparable, we have $\mathcal{T}_1\subseteq \mathcal{T}_2$ or $\mathcal{T}_1\supseteq \mathcal{T}_2$.

As we are going to see, every clique in a graph induces a tangle. We will use this fact to prove the existence of tangles in sets of data points obtained from Gaussian mixtures.

\begin{notation}
    Let $G=(V,E,w)$ be a weighted graph and suppose $W \subseteq V$. If $E(W) = \emptyset$, set $w_W \coloneqq 1$. Otherwise, let $w_{W} \coloneqq \min\{w(u,v)\mid\{u,v\}\in E(W)\}$. We set
    \begin{equation*}
        \CT_{G}(W) \coloneqq \left\{S\subseteq V \ \bigg\vert \ \kappa_G(S)<\frac{2\cdot \size{W}^2}{9}\cdot w_{W} \text{ and } \size{S\cap W}>\size{W\setminus S} \right\} \ .
    \end{equation*}
\end{notation}

The following lemma provides sufficient conditions for $\CT_{G}(W)$ to be a tangle. Note that the factor $\frac{2}{9}$ is necessary to ensure a non-empty intersection of all triples of sets contained in the tangle.

\begin{lemma}\label{lem:clique-tangle}
    Let $G=(V,E,w)$ be a weighted graph. Let $W \subseteq V$ such that $\size{W} \geq 2$ and $G[W]$ is a clique. For $w_{W} \coloneqq \min\{w(u,v)\mid\{u,v\}\in E(W)\}$, it holds that
      $\CT_{G}(W)$ is a $\kappa_G$-tangle of order $\frac{2}{9} \cdot \size{W}^2 \cdot w_{W}$.
\end{lemma}

\begin{proof}
    Let $k \coloneqq |W|$ and $K \coloneqq G[W]$.
    (T.0) and (T.3) hold trivially by the definition of $\CT_{G}(W)$. (T.1) holds, since for every vertex set $S$ that contains exactly half of all vertices of $W$, we have $\kappa_G(S)\geq\kappa_{K}(S\cap W)\geq(\frac{k}{2})^2 \cdot w_W>\frac{2k^2}{9}\cdot w_W$. To see that (T.2) holds, we show that $\kappa_G(S)<\frac{2k^2}{9} \cdot w_W$ implies that $S$ or $\compl{S}$ contains more than $\frac{2}{3}$ of all vertices in~$W$. We again use that for every $S\subseteq V$, it holds that $\kappa_G(S)>\kappa_{K}(S\cap W)$, thus we only need to compute $\kappa_{K}(S\cap W)$. Assume there is some $S\subset V$ such that $\kappa_{K}(S\cap W)<\frac{2k^2}{9}w_W$ and $\frac{1}{2}\size{W}<\size{S\cap W}\leq\frac{2}{3}\size{W}$. We compute $\kappa_{K}(S\cap W)\geq\size{S\cap W}\cdot \size{W\setminus S} \cdot w_W = \size{S\cap W} \cdot (k-\size{S\cap W}) \cdot w_W$ and see that the right-hand side becomes minimal when $\size{S\cap W}=\lfloor\frac{2}{3}k\rfloor$. For this minimum, we obtain the bound $\kappa_{K}(S\cap W)\geq\frac{2k^2}{9}w_W$. However, this contradicts the assumption $\kappa_{K}(S\cap W)<\frac{2k^2}{9}w_W$. Thus, we can tighten the restriction $\size{S\cap W}>\size{W\setminus S}$ to $\size{S\cap W}>\frac{2}{3}k$. Then (T.2) holds trivially.
\end{proof}

\subsection{Probability theory}

We give a brief introduction to the main concepts and notation we use from probability theory. For more background, we refer the reader to \cite{klen14}. As usual, we assume our data to be a subset of some
real vector space $\RR^d$, which we equip with the Euclidean norm $\lVert \cdot  \rVert$ as well as the corresponding (standard) topology and Borel $\sigma$-algebra, whose elements we call \emph{Borel sets}. For $\vec{x} \in \RR^d$ and $\delta \geq 0$, we define $B_\delta[\vec{x}] \coloneqq \{ \vec{x'}  \mid \vec{x'} \in \RR^d, \lVert \vec{x'}-\vec{x} \rVert \leq \delta\}$. For a Borel set $S\subseteq \RR^d$, we denote its complement $\RR^d\setminus S$ by $\compl{S}$ and its indicator function by $\mathds{1}_S$. We also set $B_{\delta}[S] \coloneqq \bigcup_{\vec{x} \in S} B_{\delta}[\vec{x}]$.\footnote{The set $B_{\delta}[S]$ is not necessarily a Borel set. However, it is Lebesgue-measurable as a continuous image of the Borel set $S \times B_{\delta}(\vec{0})$, which suffices for our purposes.} The probability spaces we consider consist of some $\RR^d$, its Borel $\sigma$-algebra and a probability distribution $\DD$. Instead of $\DD(\{D \in \RR^d \colon D \text{ has property }E\})$, we write $\Pr_{D \sim \DD}(E)$ or simply $\Pr(E)$.

All random variables that we consider are real-valued. For a random variable $X$, we denote its expected value by $\EX(X)$ and its variance by $\Var(X)$. The covariance of variables $X$ and $Y$ is denoted by $\Cov(X,Y)$. Random variables $X_1,\ldots,X_n$ are called \emph{independent} if, for all $A_1, \ldots, A_n \in \mathcal{B}(\RR^d)$, it holds that
$\Pr(\bigwedge_{i=1}^n (X_i \in A_i)) = \prod_{i=1}^n \Pr(X_i \in A_i)$.

We often use the inequality $\Pr(A \wedge B) \geq \Pr(A) + \Pr(B) - 1$, which is equivalent to the well-known \emph{union bound}.

Recall that the \emph{Gaussian distribution} $\mathcal{N}(\mu, \sigma^2)$ with mean $\mu\in\RR $ and standard deviation $\sigma>0$ as parameters is a distribution on $\RR$ with density 
$\phi(x \mid \mu,\sigma^2)=\frac{1}{\sigma\sqrt{2\pi}}\exp\Big({-\frac{(x-\mu)^2}{2\sigma^2}}\Big)$.
Its distribution function is denoted by $\Phi_{\mu, \sigma^2}$. For $\vec{\mu} = (\mu_1, \ldots, \mu_d)\in\RR^d$ and $\sigma > 0$, the $d$-dimensional \emph{spherical Gaussian distribution} $\mathcal{N}^d_{\text{sp}}(\vec\mu, \sigma^2)$ with mean $\vec{\mu}$ and standard deviation $\sigma$ is the distribution of the random vector $(X_1, \ldots,X_d)$ whose coordinates are independent and the $i$-th coordinate is $\mathcal{N}(\mu_i, \sigma^2)$-distributed.

We recall the Bienaymé-Chebyshev inequality.
\begin{theorem}[Bienaymé-Chebyshev's inequality \cite{bien53,che67}]    \label{thm:chebyshev_app}
    Let $X$ be a random variable with finite expected value and variance. Then, for each $t > 0$, it holds that
    \[
    \Pr\big(\size{X - \EX(X)} \geq t\big) \leq \frac{\Var(X)}{t^2}.
    \]
\end{theorem}

The next inequalities study the distribution of a sum of independent random variables.
We use the following special case of the well-known Hoeffding's inequality  \cite{hoeff63}.

\begin{theorem}[Hoeffding's Inequality \cite{hoeff63}]
Let $X_1, \ldots, X_n$ be independent random variables such that for each $i \in [n]$, there exist real numbers $a_i$ and $b_i$ with $a_i \leq X_i \leq b_i$. Let $\mean{X} \coloneqq \frac{1}{n}\sum_{i=1}^n X_i$. If $\EE(\mean{X}) < 0$, then
\[
\Pr\left(\sum_{i=1}^n X_i < 0\right) > 1 - \exp\left(-\frac{2n^2\EE(\mean{X})^2}{\sum_{i=1}^n (b_i - a_i)^2}\right).
\]
\label{cor:hoeffding}
\end{theorem}

The following theorem is a quantitative version of the central limit theorem.
\begin{theorem}[Berry-Esseen's Inequality (\cite{ber41,ess45})]
    Let $X_1, \ldots, X_n$ be independent random variables, each with finite expected value, finite variance and finite absolute third central moment $\rho_i = \EX(\size{X_i - \EX(X_i)}^3)$. For $x \in \mathbb{R}$, let 
    \[
    F_n(x) \coloneqq \Pr\left(\frac{\sum_{i=1}^n X_i - \sum \EX(X_i)}{\sqrt{\sum_{i=1}^n \Var(X_i)}} \leq x\right).
    \]
    Then there is a constant $C \in \mathbb{R}$ such that
    $
    \smash{\size{F_n(x) - \Phi_{0,1}(x)} \leq \frac{C \cdot \sum_{i=1}^n \rho_i}{\left(\sum_{i=1}^n \Var(X_i)\right)^\frac{3}{2}}} \ .
    $
    \label{thm:berry-esseen}
\end{theorem}

The currently best value for the constant in the theorem is $C = 0.5591$ \cite{tyu12}.

\section{The data model}
\label{sec:datamodel}
We assume that the data points are drawn from two or more spherical Gaussian distributions (which we call the \emph{marginals} or \emph{marginal distributions}) with given parameters.\footnote{The task of learning these parameters from a set of data is very well studied, see e.g. \cite{LiuM21} or \cite{MoitraV10}.} In most of the considered cases, we will have two distributions with equal standard deviations. Let $d$ be the dimension of the marginal distributions. Then the data points will be represented by $d$-dimensional vectors. To make explicit reference to them, it will be convenient to consider these vectors as labeled with indices $1, \dots, n$. Therefore, we assume our data encoded as a $(d \times n)$-dimensional matrix $D$ whose $i$-th column is the data point $\vec{x}_i$. Accordingly, we also consider $D$ as a transposed tuple of vectors and write $D = (\vec{x}_1, \dots, \vec{x}_n)$. 

In the general situation, we cannot assume to draw the same number of data points from each of the participating Gaussian distributions.
Instead, we assume that, for each marginal, the number of data points drawn from it is given as a parameter.\footnote{The most common alternative model is drawing the data from a Gaussian mixture with given weights (see, e.g., \cite[Chapter~8]{hennig2015handbook}).
This is mathematically easier to handle due to the independence of the data points. Our results thus translate to this setting. However, the data will have a slightly higher variance, which leads to stronger requirements for our theorems to be applicable.} To this end, we introduce a \emph{hidden labeling}, which assigns to each matrix column the distribution that the corresponding data point is drawn from. To consider variable total numbers of data points, we fix the ratios of points drawn from each distribution. 

\begin{definition}
    Let $m \in \NN$ and $r_1, \ldots, r_m \in \mathbb{Q}_{>0}$ and suppose that $r_1 + \ldots + r_m = 1$. A natural number $n$ is called \emph{compatible with $(r_1, \ldots, r_m)$} if all $r_k \cdot n$ are in $\NN$. In that case, we set $n_k \coloneqq r_k \cdot n$. A \emph{hidden labeling for $(r_1, \dots, r_m, n)$} is a function $\ell \colon [n] \to [m]$ such that, for each $k \in [m]$, there are exactly $n_k$ indices $i \in [n]$ that are mapped to $k$.
\end{definition}

If we write $r_k = \frac{p_k}{q_k}$ with coprime natural numbers $p_k$ and $q_k$, then $n$ is compatible with $(r_1, \ldots, r_m)$ if and only if it is divisible by all the $q_k$.  If $r_1, \ldots, r_m$ are clear from the context, we drop them and simply speak about \emph{compatible} $n$ and $\ell$.

\begin{definition}
    Suppose $d, m \in \NN$ and $r_1, \ldots, r_m \in \mathbb{Q}_{>0}$ with $r_1 + \ldots + r_m = 1$ and let $n$ be compatible with $(r_1, \ldots, r_m)$. Let $\ell \colon [n] \rightarrow [m]$ be a hidden labeling for $(r_1, \dots, r_m, n)$.
    Furthermore, let $(\vec{\mu}_1, \sigma_1), \ldots, (\vec{\mu}_m, \sigma_m)$ be such that for each $k \in [m]$, it holds that $\vec{\mu}_k \in \RR^d$ and $\sigma_k > 0$. Let $(X_{i})_{i \in [n]}$ be a collection of independent $d$-dimensional random variables such that $X_{i} \sim \mathcal{N}^d_{\text{sp}}(\vec{\mu}_{\ell(i)},\sigma_{\ell(i)}^2)$. 
    The \emph{data distribution defined by $((r_k, \vec{\mu}_k, \sigma_k)_{k \in [m]}, n, \ell)$} is the joint distribution of the $X_{i}$ and denoted by $\mathbb{D}[(r_k, \vec{\mu}_k, \sigma_k)_{k \in [m]},n,\ell]$.
    \label{def:dataset}
\end{definition}

Note that the induced data distribution is a distribution over the real vector space of $(d \times n)$-dimensional matrices whose colums correspond to $d$-dimensional data points.

\begin{notation}\label{not:density}
    In the setting of \Cref{def:dataset}, we let $f_k$ be the density of the $k$-th Gaussian distribution; that is, $f_k(\vec{x}) = \phi(\vec{x} \mid \vec{\mu}_k, \sigma_k^2)$. The induced measure is denoted by $\nu_k$, so for every Borel set $A\subseteq \RR^d$, we have $\nu_k(A) = \int_{A} f_k(\vec{x}) d\vec{x}$. We set $\mean{f} \coloneqq \sum r_k f_k$ and $\mean{\nu} \coloneqq \sum r_k \nu_k$.
\end{notation}
Note that in the described setting, for all compatible $n$ and $\ell$, it holds that $\frac{1}{n}\sum_{i=1}^n f_{\ell(i)} = \mean{f}$.

To analyze the data points, we define an underlying graph structure, which captures the similarity between the data points. More precisely, we interpret the data points as the vertices and consider two different ways to define (weighted) adjacency, both of which take the Euclidean distance between vertices into account. 

\begin{definition}
    Let $n, d \in \NN_{>0}$ and suppose
    $w \colon \RR^d \times \RR^d \to \RR_{\geq 0}$ is symmetric and $D = (\vec{x}_i)_{i \in [n]}\in \RR^{d \times n}$.
    Then $G(D,w)$ denotes the weighted graph with vertex set $[n]$, edge set $\big\{\{i,j\}\mid w(\vec{x}_i, \vec{x}_j) > 0\big\}$ and weights $w'(i,j) = w(\vec{x}_i, \vec{x}_j)$.
\end{definition}

In the first model, we maintain all distances. We define a weighted graph, which assigns to the edge between a pair of vertices a weight that is monotonous in their distance and upper-bounded by $1$ with equality if and only if the distance between the data points is $0$.

\begin{example}
\label{ex:completegraph}
    For $c > 0$ and $\hat{w}_{c}(\vec{x},\vec{y}) \coloneqq \exp\left(\frac{- \lVert \vec{x} - \vec{y} \rVert^2}{2c^2}\right)$, the graph $G(D, \hat{w}_c)$ is a \emph{\completegraph}.
     In the context of this work, we will only have $c = \sigma$, where $\sigma$ is the standard deviation of the Gaussian distribution at hand (see also \cite{luxburg2007spectral}). Accordingly, we drop $c$ and will simply speak about \emph{the} \completegraph. See also \Cref{fig:datamodel_intuition3}.
    \end{example}

Instead of maintaining all distances between pairs of vertices, we can also define a threshold for adjacency and forget the distances between vertices that are further apart.

\begin{example}
    For $\delta > 0$ and $w_{\delta}(\vec{x},\vec{y}) \coloneqq \begin{cases}
    1, \text{ if } \lVert \vec{x} - \vec{y} \rVert \leq \delta \\
    0, \text{ otherwise}\end{cases},
    $
    the graph $G(D, w_{\delta})$ is a \emph{\deltagraph}. See also \Cref{fig:datamodel_intuition2}.
    \label{ex:deltagraph}
\end{example}

Both of these models are standard in the clustering context \cite{luxburg2007spectral}\footnote{Besides the graph models considered here, the article \cite{luxburg2007spectral} lists a third one. There, a vertex is adjacent to its $k$ nearest neighbors. However, the model is not suitable for our context, because the existence of an edge between two vertices does not only depend on the positions of the corresponding data points.}. Since every graph vertex represents a point in $\RR^d$, we can interpret separations of $\RR^d$ as separations of the vertex set in the \deltagraph\ and the \completegraph. This motivates the following notation and will ultimately allow us to bound  probabilities for events on the graphs by using bounds on probabilities in $\RR^d$ associated with the data distributions we consider.

\begin{notation}\label{not:vertices:in:a}
    For $d,n \in \NN_{> 0}$, $A \subseteq \RR^d$ and a $(d \times n)$-dimensional data matrix $D$, we set $V_A(D) \coloneqq \{i \in [n] \mid \vec{x}_i \in A\}$. If $D$ is clear from the context, we only write $V_A$.
\end{notation}

Note that, if $A \subseteq \RR^d$ is a Borel set and $\DD$ is a probability distribution on $\RR^{d \times n}$, then the size of $V_A$ (denoted by $\size{V_A}$) is a random variable. 

The following lemma gives tool insights about the dependence of $\EE(|V_A|)$ and $\Var(|V_A|)$ on $n$. We use those in the next section to obtain the desired probability bounds.

\begin{lemma}\label{lem:basic_computations_app}
	Assume $d, m \in \NN_{>0}$ and let $(r_1, \vec{\mu}_1, \sigma_1), \ldots, (r_m, \vec{\mu}_m, \sigma_m)$ be as in \Cref{def:dataset}. Let $A \subseteq \RR^d$ be a Borel set and suppose $w \in \{w_\delta, \hat{w}_c\}$.
	Then, for compatible $n$ and $\ell$ and $\mathbb{D} \coloneqq \mathbb{D}[(r_k, \vec{\mu}_k, \sigma_k)_{k \in [m]},n,\ell]$, the following statements hold.
    \begin{enumerate}
        \item \label{enum:exp_size_app} $\EX_{D \sim \DD}(\size{V_A}) = n \cdot \bar{\nu}(A)$.
        \item \label{enum:var_size_app}$\Var_{D \sim \DD}(\size{V_A})= n \cdot \sum_{k=1}^m r_k \nu_k(A) (1 - \nu_k(A))$.
        \item \label{enum:exp_size_sq_app}$\EX_{D \sim \DD}(\size{V_A}^2) = n^2 \cdot \bar{\nu}(A)^2 + n \cdot \sum_{k=1}^m r_k \nu_k(A) (1 - \nu_k(A)).$
        \item \label{enum:var_size_sq_app} $\Var_{D \sim \DD}(\size{V_A}^2) = O(n^3)$.
        \item \label{enum:exp_kappa_app} There is a $c \geq 0$ such that 
            \[\EX_{D \sim \DD}\big(\kappa_{G(D,w)}(V_A)\big) = n^2 \cdot \int_{A} \int_{\compl{A}} w(\vec{x},\vec{y}) \bar{f}(\vec{x}) \bar{f}(\vec{y}) \, d\vec{y} \, d\vec{x} - n \cdot c.
            \]
        \item \label{enum:var_kappa_app} $\Var_{D \sim \DD}(\kappa_{G(D,w)}(V_A)) = O(n^3)$.
        \item \label{enum:size_geq_two_app}$
        \Pr\limits_{D \sim \mathbb{D}}\left(\size{V_A} \geq 2\right) = 1 -  \left(1 +  \sum\limits_{k=1}^m  \frac{n_k \cdot \nu_k(A)}{1 - \nu_k(A)}\right) \cdot \prod\limits_{k=1}^m \big(1 - \nu_k(A)\big)^{n_k}$
    \end{enumerate}
\end{lemma}

\begin{proof} 
For better readability, we write $\EX$, $\Var$, $\Cov$, and $\kappa$ instead of $\EX_{D \sim \DD}$, $\Var_{D \sim \DD}$, $\Cov_{D \sim \DD}$, and $\kappa_{G(D,w)}$ throughout this proof.

To prove Parts \ref{enum:exp_size_app} and \ref{enum:var_size_app}, we observe that
    \[\EX(\size{V_A}) = \EX\left(\sum_{i=1}^n \mathds{1}_A(X_i)\right) = \sum_{i=1}^n \EX(\mathds{1}_A(X_i)) = \sum_{i=1}^n \nu_{\ell(i)}(A) = n \cdot \bar{\nu}(A)\]
and, since the $X_i$ are independent,
    \begin{align*}
        \Var(\size{V_A}) &= \Var\left(\sum_{i=1}^n \mathds{1}_A(X_i)\right) = \sum_{i=1}^n \Var(\mathds{1}_A(X_i)) \\
        &= \sum_{i=1}^n \nu_{\ell(i)}(A) (1 - \nu_{\ell(i)}(A)) = \sum_{k=1}^m n_k \nu_k(A) (1 - \nu_k(A)) \\
        &= n \cdot \sum_{k=1}^m r_k \nu_k(A) (1 - \nu_k(A)).
    \end{align*}
Now Part \ref{enum:exp_size_sq_app} immediately follows from the well-known equality
$
\EX(\size{V_A}^2) = \EX(\size{V_A})^2 + \Var(\size{V_A})
$.

For Part \ref{enum:var_size_sq_app}, we expand $\size{V_A}^2 = (\sum_{i=1}^n \mathds{1}_A(X_i))^2 = \sum_{i=1}^n\sum_{j=1}^n \mathds{1}_A(X_i)\cdot\mathds{1}_A(X_j)$ and conclude
\begin{equation}
\Var(\size{V_A}^2) = \sum_{i=1}^n\sum_{j=1}^n \sum_{i'=1}^n\sum_{j'=1}^n \Cov\big(\mathds{1}_A(X_i)\cdot\mathds{1}_A(X_j), \mathds{1}_A(X_{i'})\cdot\mathds{1}_A(X_{j'})\big).
\label{eqn:var_from_sum_of_cov_app}
\end{equation}
If $i,j,i',j'$ are pairwise distinct, then the random variables $X_i, X_j, X_{i'}$ and $X_{j'}$ are independent and hence $\Cov(\mathds{1}_A(X_i)\cdot\mathds{1}_A(X_j), \mathds{1}_A(X_{i'})\cdot\mathds{1}_A(X_{j'})) = 0$. This implies that $n(n-1)(n-2)(n-3)$ of the $n^4$ summands in \Cref{eqn:var_from_sum_of_cov_app} are equal to $0$. Since every other summand is also at most $1$, this shows the claim.

For proving Parts \ref{enum:exp_kappa_app} and \ref{enum:var_kappa_app}, we set $Y_{i,j} \coloneqq w(X_i,X_j) \cdot \mathds{1}_{A}(X_i) \cdot \mathds{1}_{\compl{A}}(X_j)$, which yields $Y_{i,i} = 0$ whenever $i = j$ and otherwise, for $i \neq j$,
\[
\EX(Y_{i,j}) = \int_{A} \int_{\compl{A}} w(\vec{x},\vec{y}) f_{\ell(i)}(\vec{x}) f_{\ell(j)}(\vec{y}) \, d\vec{y} \, d\vec{x}\ .
\]
Since $
\kappa(V_A) = \sum_{i=1}^n \sum_{j=1}^n Y_{i,j},
$
we conclude 
\begin{align*}
    &\int_{A} \int_{\compl{A}} w(\vec{x},\vec{y}) \bar{f}(\vec{x}) \bar{f}(\vec{y}) \, d\vec{y} \, d\vec{x} \\
    = &\int_{A} \int_{\compl{A}} w(\vec{x},\vec{y}) \bigg(\frac{1}{n}\sum_{i=1}^n f_{\ell(i)}(\vec{x})\bigg) \bigg(\frac{1}{n}\sum_{j=1}^n f_{\ell(j)}(\vec{y})\bigg) \, d\vec{y} \, d\vec{x} \\
    = &\frac{1}{n^2} \int_{A} \int_{\compl{A}} w(\vec{x},\vec{y}) \bigg(\sum_{i=1}^n\sum_{j=1}^n f_{\ell(i)}(\vec{x})f_{\ell(j)}(\vec{y})\bigg) \, d\vec{y} \, d\vec{x} \\
    = &\frac{1}{n^2} \sum_{i=1}^n\sum_{j=1}^n \int_{A} \int_{\compl{A}} w(\vec{x},\vec{y}) f_{\ell(i)}(\vec{x})f_{\ell(j)}(\vec{y}) \, d\vec{y} \, d\vec{x} \\
    = &\frac{1}{n^2} \sum_{i=1}^n\sum_{\substack{j=1 \\ j \neq i}}^n \int\limits_{A} \int\limits_{\compl{A}}  w(\vec{x},\vec{y}) f_{\ell(i)}(\vec{x})f_{\ell(j)}(\vec{y}) \, d\vec{y} \, d\vec{x}\\
    &+ \frac{1}{n^2} \mspace{-2mu}\sum_{i=1}^n \int\limits_{A} \int\limits_{\compl{A}}  w(\vec{x},\vec{y}) f_{\ell(i)}(\vec{x})f_{\ell(i)}(\vec{y}) \, d\vec{y} \, d\vec{x}\\
    =& \frac{1}{n^2} \sum_{i=1}^n\sum_{j=1}^n \EX(Y_{i,j}) + \frac{1}{n^2} \sum_{k=1}^m n_k \cdot \int_{A} \int_{\compl{A}} w(\vec{x},\vec{y}) f_k(\vec{x}) f_k(\vec{y}) \, d\vec{y} \, d\vec{x}\\
    =& \frac{1}{n^2}\EX(\kappa(V_A)) + \frac{1}{n} \bigg(\sum_{k=1}^m r_k \int_{A} \int_{\compl{A}} w(\vec{x},\vec{y}) f_k(\vec{x})f_k(\vec{y}) \, d\vec{y} \, d\vec{x} \bigg) \ .
\end{align*}
Setting $c \coloneqq \sum_{k=1}^m r_k \int_{A} \int_{\compl{A}} w(\vec{x},\vec{y}) f_k(\vec{x})f_k(\vec{y}) \, d\vec{y} \, d\vec{x}$, this yields Part \ref{enum:exp_kappa_app}.

Next, we expand
\begin{equation}
\Var(\kappa(V_A)) = \Var(\sum_{i=1}^n\sum_{j=1}^n Y_{i,j}) = \sum_{i=1}^n\sum_{j=1}^n \sum_{i'=1}^n\sum_{j'=1}^n \Cov(Y_{i,j}, Y_{i',j'})
\label{eqn:var_kappa_from_sum_of_cov_app}    
\end{equation}
and observe for pairwise distinct $i,j,i',j'$ that $X_i, X_j, X_{i'}, X_{j'}$ are independent, which in turn yields the independence of $Y_{i,j}$ and $Y_{i',j'}$. Thus, we get that $\Cov(Y_{i,j}, Y_{i',j'}) = 0$. This yields that there are $O(n^3)$ non-zero summands in \Cref{eqn:var_kappa_from_sum_of_cov_app} and it remains to show that these can be bounded independently from $n$ and $\ell$. Indeed, since the value of $\Cov(Y_{i,j}, Y_{i',j'})$ is determined by $\ell(i),\ell(j),\ell(i'),\ell(j')$ and the information which of the indices are equal, we can find a common upper bound on the covariances that only depends on $(r_1, \vec{\mu}_1, \sigma_1), \ldots, (r_m, \vec{\mu}_m, \sigma_m)$, $A$ and $w$.
This yields Part \ref{enum:var_kappa_app}.

Finally, to prove the formula for the probability of $\size{V_A} \geq 2$, we write
\[
    \Pr_{D \sim \DD}\big(\size{V_A} \geq 2\big) = 1 - \Pr_{D \sim \DD}\big(\size{V_A} = 0\big) - \Pr_{D \sim \DD}\big(\size{V_A} = 1\big)
\]
and use the independence of the random variables $X_i$ to obtain
\begin{align*}
    \Pr_{D \sim \DD}\big(\size{V_A} = 0\big) = \Pr_{D \sim \DD}\Big(\bigwedge_{i=1}^n(X_{i}\notin A)\Big) = \prod_{i=1}^n \Pr_{D \sim \DD}\big(X_i \notin A\big) = \prod_{i=1}^n (1 - \nu_{\ell(i)}(A))
\end{align*}
and, since all $\nu_k(A) < 1$,
\begin{align*}
    \Pr_{D \sim \DD}\big(\size{V_A} = 1\big) &= \sum_{i=1}^n\Pr_{D \sim \DD}\Big((X_i \in A) \wedge \bigwedge_{\substack{j = 1 \\ j \neq i}}^n (X_j\notin A)\Big) \\
    &=\sum_{i=1}^n\Pr_{D \sim \DD}(X_i \in A) \cdot \prod_{\substack{j = 1 \\ j \neq i}}^n \Pr_{D \sim \DD}(X_j\notin A) \\
    &= \sum_{i=1}^n\nu_{\ell(i)}(A) \cdot \prod_{\substack{j = 1 \\ j \neq i}}^n (1-\nu_{\ell(j)}(A))\\
    &= \sum_{i=1}^n\frac{\nu_{\ell(i)}(A)}{1 - \nu_{\ell(i)}(A)} \cdot \prod_{j = 1}^n (1-\nu_{\ell(j)}(A)).
\end{align*}
Together, this yields
\begin{align*}
    \Pr_{D \sim \DD}\big(\size{V_A} \geq 2\big) &= 1 - \left(1 + \sum_{i=1}^n\frac{\nu_{\ell(i)}(A)}{1 - \nu_{\ell(i)}(A)}\right) \cdot \prod_{j = 1}^n (1-\nu_{\ell(j)}(A)) \\
    &= 1 -  \left(1 +  \sum_{k=1}^m  \frac{n_k \cdot \nu_k(A)}{1 - \nu_k(A)}\right) \cdot \prod_{k=1}^m (1 - \nu_k(A))^{n_k} \ ,
\end{align*}
which shows Part \ref{enum:size_geq_two_app}.
\end{proof}

\section{Results for the neighborhood graph}
\label{sec:delta}
In this section, we study the tangles that occur in the \deltagraph, as defined in \Cref{ex:deltagraph}. We find that the unweighted structure of the graph makes it possible to derive strong lower bounds on the probability that incomparable tangles corresponding to the different marginal Gaussian distributions exist. 
By Lemma \ref{lem:clique-tangle}, every non-trivial clique $G[W]$ induces a $\edgecon{G}$-tangle. Every hyperball of radius at most $\frac{\delta}{2}$ induces a clique in the \deltagraph.
As our goal is to identify tangles associated with the dense regions around the means, we consider the $\frac{\delta}{2}$-neighborhood $B_{\frac{\delta}{2}}[\vec{\mu}]$ of $\mu$ in the following. 
We bound the probability that the induced cliques are non-trivial and that the tangles corresponding to different means are indeed incomparable. To do so, we find small separations that are oriented differently by the different tangles.

\subsection{Probability bounds}
\label{sec:probability_delta}

To compute the probability bounds for the existence of incomparable tangles, we first give a bound on the probability that, for a given separation, the $\frac{\delta}{2}$-ball around the mean of a distribution induces a tangle with a prescribed orientation of the separation. 

To consider different geometric objects for the separation, we define it via a Borel set $S$. In \Cref{subsec:comp}, we consider particular choices for $S$ that fulfill the requirements to obtain the probability bounds presented in this section. Often the local minima of the underlying joint distribution will induce a good choice for the Borel set $S$, but, to actually compute the probabilities, it can be useful to make different choices.

\begin{theorem}\label{thm:large_n_delta}
    Suppose $\delta > 0$, $d, m \in \NN_{>0}$ and let $(r_1, \vec{\mu}_1, \sigma_1), \ldots, (r_m, \vec{\mu}_m, \sigma_m)$ be as in \Cref{def:dataset}. Choose a $k \in [m]$ and define $B \coloneqq B_{\frac{\delta}{2}}[\vec{\mu_k}]$. Let $S \subseteq \RR^d$ be a Borel set.
    
    If $\mean{\nu}(B \cap S) > \frac{1}{2} \mean{\nu} (B)$ and
        $\frac{2}{9} \mean{\nu}(B)^2 > \int_{S} \int_{\compl{S}} w_{\delta}(\vec{x},\vec{y}) \mean{f}(\vec{x}) \mean{f}(\vec{y}) \, d\vec{y} \, d\vec{x}$,
    then for compatible $n$ and $\ell$ and $\mathbb{D} = \mathbb{D}[(r_k, \vec{\mu}_k, \sigma_k)_{k \in [m]},n,\ell]$, it holds that
    \[
    \Pr_{D \sim \mathbb{D}}\left( \CT_{G(D, w_{\delta})}(V_B) \text{ is a tangle that contains } V_S
    \right) = 1 - O\left(\tfrac{1}{n}\right) \ .
    \]
\end{theorem}

\begin{proof}
For better readability, let $G \coloneqq G(D, w_{\delta})$ and $\kappa \coloneqq \kappa_G$.
 Since $B$ is a ball with radius $\frac{\delta}{2}$, every two data points in $B$ have distance at most $\delta$, so $V_B$ induces a clique in $G$ with $w_{V_B} = 1$ where  $w_{V_B} \coloneqq \min\{w(u,v)\mid\{u,v\}\in E(V_B)\}$.
 Note that, by \Cref{lem:clique-tangle}, it holds that $\CT_G(V_B)$ is a tangle as soon as $\size{V_B}\geq 2$. Furthermore, for the tangle to have the desired orientation, i.e., for $V_S$ to be a member of $\CT_{G}(V_B)$, the inequalities $\size{V_B \cap V_S} > \size{V_B \setminus V_S}$ and $\frac{2}{9}\size{V_B}^2 > \kappa(V_S)$ must hold. 

In the following, we define four events which will ultimately enable us to deduce a lower bound on the probability that the conditions
\begin{enumerate}
    \item\label{it:ineq1_app} $\size{V_B}\geq 2$,
    \item\label{it:ineq2_app} $\size{V_B \cap V_S} > \size{V_B \setminus V_S}$ \text{ and}
    \item\label{it:ineq3_app} $\frac{2}{9}\size{V_B}^2 > \kappa(V_S)$
\end{enumerate}
hold simultaneously. Note that the assumptions of the theorem ensure that Inequalities \eqref{it:ineq2_app} and \eqref{it:ineq3_app} hold in expectation. Indeed, 
\begin{align*}
\EX(\size{V_B \cap V_S}) 
 &\overset{\text{L.\ref{lem:basic_computations_app}.\ref{enum:exp_size_app}}}{=}
n \cdot \mean{\nu}(B \cap S) 
> 
n \cdot (\mean{\nu}(B) - (\mean{\nu}(B \cap S)) \\
&= 
n \cdot \mean{\nu}(B \setminus S) 
\overset{\text{L.\ref{lem:basic_computations_app}.\ref{enum:exp_size_app}}}{=}
\EX(\size{V_B \setminus V_S})
\end{align*}

and similarly, by \Cref{lem:basic_computations_app}, Parts \ref{enum:exp_size_sq_app} and  \ref{enum:exp_kappa_app}, Inequality \eqref{it:ineq3_app} holds in expectation. 
Therefore, we can use the Bienaymé-Chebyshev inequality (\Cref{thm:chebyshev_app}) for each of the random variables $\size{V_{B \cap S}}$, $\size{V_{B \setminus S}}$, $\size{V_B}^2$,
and $\kappa(V_S)$ to derive probability bounds. Set 

\begin{multline*}
\varepsilon_1 \coloneqq \frac{1}{2}\cdot\big(\mean{\nu}(B \cap S) - \mean{\nu}(B \setminus S)\big)
\qquad \text{ and }\\
\epsilon_2 \coloneqq \frac{1}{2} \cdot \left(\frac{2}{9} \mean{\nu}(B)^2 - \int_{S} \int_{\compl{S}} w_{\delta}(\vec{x},\vec{y}) \mean{f}(\vec{x}) \mean{f}(\vec{y}) \, d\vec{y} \, d\vec{x}\right).
\end{multline*}

With this, we are ready to define our events $E_1, E_2, E_3, E_4$ and deduce lower bounds on their probabilities.
\begin{align*}
    \Pr_{D \sim \DD}\Big(\underbrace{\big\lvert\size{V_{B \cap S}} - \EX(\size{V_{B \cap S}}) \big\rvert < \varepsilon_1 \cdot n}_{\eqqcolon{} E_1}\Big) 
    &> \mathrlap{1 - \frac{\Var(\size{V_{B \cap S}})}{\varepsilon_1^2 n^2}}\phantom{1 - \frac{\Var(\kappa(V_S))}{\varepsilon_2^2 n^4} }
    \overset{\text{L.\ref{lem:basic_computations_app}.\ref{enum:var_size_app}}}{=}
    1 - O\left(\frac{1}{n}\right),\\
    \Pr_{D \sim \DD}\Big(\underbrace{\big\lvert\size{V_{B \setminus S}} - \EX(\size{V_{B \setminus S}}) \big\rvert < \varepsilon_1 \cdot n}_{\eqqcolon{} E_2}\Big) 
    &> \mathrlap{1 - \frac{\Var(\size{V_{B \setminus S}})}{\varepsilon_1^2 n^2}}\phantom{1 - \frac{\Var(\kappa(V_S))}{\varepsilon_2^2 n^4} }
    \overset{\text{L.\ref{lem:basic_computations_app}.\ref{enum:var_size_app}}}{=}
    1 - O\left(\frac{1}{n}\right),\\
    \Pr_{D \sim \DD}\Big(\underbrace{\big\lvert\size{V_{B}}^2 - \EX(\size{V_{B}}^2) \big\rvert < \varepsilon_2 \cdot n^2}_{\eqqcolon{} E_3}\Big) 
    &> \mathrlap{1 - \frac{\Var(\size{V_{B}}^2)}{\varepsilon_2^2 n^4}}\phantom{1 - \frac{\Var(\kappa(V_S))}{\varepsilon_2^2 n^4} }
    \overset{\text{L.\ref{lem:basic_computations_app}.\ref{enum:var_size_sq_app}}}{=}
    1 - O\left(\frac{1}{n}\right),\\
    \Pr\limits_{D \sim \DD}\Big(\underbrace{\big\lvert\kappa(V_S) - \EX(\kappa(V_S)) \big\rvert < \varepsilon_2 \cdot n^2}_{\eqqcolon{} E_4}\Big) 
    &> 1 - \frac{\Var(\kappa(V_S))}{\varepsilon_2^2 n^4} 
    \overset{\text{L.\ref{lem:basic_computations_app}.\ref{enum:var_kappa_app}}}{=} 
    1 - O\left(\frac{1}{n}\right).
\end{align*}
The union bound then yields 
\[\Pr_{D \sim \DD}(E_1 \wedge E_2 \wedge E_3 \wedge E_4) \geq \Pr_{D \sim \DD}(E_1) +\Pr_{D \sim \DD}(E_2) + \Pr_{D \sim \DD}(E_3) + \Pr_{D \sim \DD}(E_4) - 3 = 1 - O\left(\frac{1}{n}\right).
\]
It suffices now to show that $E_1 \land E_2 \land E_3 \land E_4$ entails the validity of Inequalities \eqref{it:ineq1_app}-\eqref{it:ineq3_app}. Suppose that all of $E_1, E_2, E_3, E_4$ hold. Then for $n \geq \frac{2}{\varepsilon_1}$, we have
\begin{align*}
\size{V_B} &\geq \size{V_{B\cap S}} \\
&= 
(\size{V_{B\cap S}} - \EX(\size{V_{B\cap S}}))
+ 
(\EX(\size{V_{B\cap S}}) - \EX(\size{V_{B\setminus S}}))
+ \EX(\size{V_{B\setminus S}})
\\&\overset{\mathclap{E_1}}{{}>{}}  {-}\varepsilon_1 \cdot n + 2\varepsilon_1 \cdot n + \EX(\size{V_{B\setminus S}})
\\&> \varepsilon_1 \cdot n 
\\&\geq 2 \ ,
\end{align*}
which shows Inequality \eqref{it:ineq1_app}. Moreover,
\begin{align*}
    \size{V_B \cap V_S} - \size{V_B \setminus V_S} 
    ={} &\big(\size{V_B \cap V_S} - \EX(\size{V_B \cap V_S})\big)  \\& {}+ \big(\EX(\size{V_B \cap V_S})- \EX(\size{V_B \setminus V_S})\big)  \\& {}+ \big(\EX(\size{V_B \setminus V_S}) - \size{V_B \setminus V_S}\big) \\
    ={} &\big(\size{V_{B \cap S}} - \EX(\size{V_{B \cap S}})\big) + n \cdot \big(\mean{\nu}(B \cap S) - \mean{\nu}(B \setminus S)\big)
    \\&{}+\big(\EX(\size{V_{B \setminus S}}) - \size{V_{B \setminus S}}\big)\\
    \overset{\mathclap{E_1,E_2}}{{}>{}} &{-}\varepsilon_1 \cdot n + 2 \varepsilon_1 \cdot n -\varepsilon_1 \cdot n\\ 
    ={} &0 \ ,
\end{align*}
which proves Inequality \eqref{it:ineq2_app}. Finally, by Parts \ref{enum:exp_size_sq_app} and \ref{enum:exp_kappa_app} of \Cref{lem:basic_computations_app}, there exists a $c \geq 0$ such that
\begin{align*}
	\frac{2}{9}\size{V_B}^2 \mspace{-2mu} - \mspace{-2mu}\kappa(V_S)
	= &\ \frac{2}{9}\big(\size{V_B}^2 \mspace{-1mu} - \mspace{-1mu} \EX(\size{V_B}^2)\big) \mspace{-2mu} + \mspace{-2mu} \left(\mspace{-1mu} \frac{2}{9}\EX(\size{V_B}^2) \mspace{-1mu} - \mspace{-1mu} \EX(\kappa(V_S))\mspace{-2mu}\right)  \\& {} +  \big(\EX(\kappa(V_S)) \mspace{-1mu} - \mspace{-1mu} \kappa(V_S)\big) \\
	\overset{\mathclap{E_3,E_4}}{{}>{}}& {-}\varepsilon_2 \cdot n^2 + 2\varepsilon_2 \cdot n^2 + \frac{2}{9} n \cdot \sum_{k=1}^m r_k \nu_k(B) (1 - \nu_k(B)) \\& {} + c\cdot n -\varepsilon_2 \cdot n^2 \\
	>&\ 0 \ ,
\end{align*}
which shows Inequality \eqref{it:ineq3_app}.
Together, this proves that for $n \geq \frac{2}{\varepsilon_1}$, the event $E_1 \land E_2 \land E_3 \land E_4$ implies that 
$\CT_{G}(V_B)$ is a tangle that contains $V_S$. This completes the proof.
\end{proof}

\begin{remark}
A small modification of the proof shows that the probability for $V_S$ to be a member of $\CT_{G(D, w_{\delta})}(V_B)$ is in $O(\frac{1}{n})$ if the second inequality holds with opposite sign; that is, if $\frac{2}{9} \mean{\nu}(B)^2 < \int_{S} \int_{\compl{S}} w_{\delta}(\vec{x},\vec{y}) \mean{f}(\vec{x}) \mean{f}(\vec{y}) \, d\vec{y} \, d\vec{x}$. However, this does not imply that there is no suitable tangle containing $V_S$, since a tangle of higher order containing $\CT_{G(D, w_{\delta})}(V_B)$ as a subset and $V_S$ as a member could still exist. 
\end{remark}

In fact, the proof arguments yield that, for every $\epsilon>0$, the order of the tangle is with probability $1 - O\left(\tfrac{1}{n}\right)$ at least $(1-\varepsilon)\cdot n^2 \cdot \frac{2}{9} \mean{\nu}(B)^2$.

\Cref{thm:large_n_delta} can be used to show the existence of incomparable tangles asymptotically almost surely. However, since the constants hidden in the asymptotic expression for the bound are large, for small numbers of data points, it only yields a trivial bound. Hence, we intend to make reasonably stricter requirements for the data distribution that allow us to find a probability bound with significantly faster convergence to $1$. As a first step, we derive an upper bound on the size of a cut induced by a set $S \subseteq \RR^d$.

\begin{lemma}    \label{lem:bound_kappa}
    Suppose $\delta > 0$, $d,n \in \NN_{>0}$, $D\in\RR^{d\times n}$ and let $S \subseteq \RR^d$ be a Borel set. We set $A\coloneqq B_{\delta}[S] \cap B_{\delta}[\compl{S}]$. Then
    $\kappa_{G(D, w_\delta)}(V_S) \leq \frac{1}{4}\size{V_A}^2$.
\end{lemma}

\begin{proof}
    Let $E$ be the edge set of $G(D, w_\delta)$ and consider an edge $\{i,j\} \in E(V_S, V_{\compl{S}})$ with $i \in V_S$ and $j \in V_{\compl{S}}$. By definition of the \deltagraph\ (\Cref{ex:deltagraph}), we know $\lVert \vec{x}_i - \vec{x}_j\rVert \leq \delta$. Hence $\vec{x}_i \in B_{\delta}[\vec{x}_j] \cap S \subseteq B_{\delta}[\compl{S}] \cap S$ and $\vec{x}_j \in B_{\delta}[\vec{x}_i] \cap \compl{S} \subseteq B_{\delta}[S] \cap \compl{S}$, which yields $(i,j) \in V_{B_{\delta}[\compl{S}] \cap S} \times V_{B_{\delta}[S] \cap \compl{S}}$.
    Since $B_{\delta}[\compl{S}] \cap S$ and $B_{\delta}[S] \cap \compl{S}$ are disjoint with $(B_{\delta}[\compl{S}] \cap S) \cup (B_{\delta}[S] \cap \compl{S}) = A$, we can use the well-known inequality $a\cdot b \leq \frac{1}{4}(a+b)^2$ for reals $a,b$ to conclude
    \[
    \kappa_{G(D, w_\delta)}(V_S) 
    \leq 
    \size{V_{B_{\delta}[\compl{S}] \cap S}}
    \cdot 
    \size{V_{B_{\delta}[S] \cap \compl{S}}}
    \leq 
    \frac{1}{4}
    \left( 
    \size{V_{B_{\delta}[\compl{S}] \cap S}} 
    + 
    \size{V_{B_{\delta}[S] \cap \compl{S}}}
    \right)^2 =  \frac{1}{4}\size{V_A}^2.
    \]
\end{proof}

Using the lemma, we can answer the question whether a cut is oriented by the considered tangle by checking if a sum of independent random variables is negative. Then, Hoeffding's and Berry-Esseen's Inequalities (Theorems \ref{cor:hoeffding} and \ref{thm:berry-esseen}) yield the following. 

\begin{theorem}    \label{thm:small_n_delta}
    Suppose $\delta>0$, $d, m \in \NN_{>0}$ and let $(r_1, \vec{\mu}_1, \sigma_1), \ldots, (r_m, \vec{\mu}_m, \sigma_m)$ be as in \Cref{def:dataset}. 
    Choose a $k \in [m]$ and define $B \coloneqq B_{\frac{\delta}{2}}[\vec{\mu_k}]$.
    Let $S$ be a Borel set with $B \subseteq S$. We set $A\coloneqq B_{\delta}[S] \cap B_{\delta}[\compl{S}]$ and suppose
    $
        2 \sqrt{2} \mean{\nu}(B) > 3\mean{\nu}(A).
    $
   
    Let $n$ and $\ell$ be compatible and $\DD = \mathbb{D}[(r_k, \vec{\mu}_k, \sigma_k)_{k \in [m]},n,\ell]$. For all $i \in [n]$, we define random variables $Y_i \coloneqq (3 \cdot \mathds{1}_{A} - 2\sqrt{2} \cdot \mathds{1}_{B})(X_i)$ and set $\rho_i \coloneqq \EX(\size{Y_i-\EX(Y_i)}^3)$. Then
    \begin{align*}
    &\Pr_{D \sim \mathbb{D}}\big( \CT_{G(D, w_{\delta})}(V_B) \text{ is a tangle that contains } V_S\big) \\
    &{}\geq{} \ \max\Bigg\{1 - \exp\left(-n \cdot \frac{2}{(2\sqrt{2} + 3)^2}\cdot(3\mean{\nu}(A) - 2 \sqrt{2} \mean{\nu}(B)\big)^2\right),\\&\phantom{{}\geq{} \ \max\bigg\}} \Phi_{0,1}\left(\frac{-\sum_{i=1}^n \EX(Y_i)}{\sqrt{\sum_{i=1}^n \Var(Y_i)}}\right) -  \frac{C \cdot\sum_{i=1}^n \rho_i}{(\sum_{i=1}^n \Var(Y_i))^{\frac{3}{2}}}\Bigg\} \\
    &\phantom{{}\geq{} \ }- \left(1 +  \sum_{k=1}^m  \frac{n_k \cdot \nu_k(B)}{1 - \nu_k(B)}\right) \cdot \prod_{k=1}^m \big(1 - \nu_k(B)\big)^{n_k}
    \end{align*}
    where $C$ is the constant from \Cref{thm:berry-esseen}.
    As a consequence, 
    \[\Pr_{D \sim \mathbb{D}}\big( \CT_{G(D, w_{\delta})}(V_B) \text{ is a tangle that contains } V_S \big) = 1 - e^{-\Omega(n)}.
    \] 
\end{theorem}

\begin{proof}
       For better readability, let $G \coloneqq G(D, w_{\delta})$ and $\EX \coloneqq \EX_{D \sim \DD}$.  As we already observed in the proof of \Cref{thm:large_n_delta}, the set $V_B$ induces a clique in $G$ with $w_{V_B} = 1$ and hence, by \Cref{lem:clique-tangle}, $\CT_G(V_B)$ is a tangle as soon as $\size{V_B}\geq 2$. Since $B \subseteq S$, the set $V_S$ is a member of $\CT_G(V_B)$ as soon as $\frac{2}{9} \size{V_B}^2 > \kappa_G(V_S)$. By \Cref{lem:bound_kappa}, this is true if $\frac{2}{9} \size{V_B}^2 > \frac{1}{4}\size{V_A}^2$ or, equivalently, $2\sqrt{2}\size{V_B} > 3\size{V_A}$.
    Hence,
    \begin{align*}
        \Pr_{D \sim \DD}&(\CT_G \text{ is a tangle that contains } V_S)\\
        &\geq 
        \Pr_{D \sim \DD}\left(2\sqrt{2} \size{V_B} > 3\size{V_A} \text{ and } \size{V_B}\geq 2\right)\\
        &=\Pr_{D \sim \DD}\left(2\sqrt{2} \size{V_B} \geq 3\size{V_A} \text{ and } \size{V_B}\geq 2\right)\\
        &\geq\Pr_{D \sim \DD}\left(3\size{V_A} \mspace{-2mu} - \mspace{-2mu} 2\sqrt{2} \size{V_B} \mspace{-1mu} \leq \mspace{-1mu} 0\right) \mspace{-2mu} + \mspace{-4mu} \Pr_{D \sim \DD}(\size{V_B} \mspace{-1mu} \geq \mspace{-1mu} 2) \mspace{-2mu} - \mspace{-2mu} 1 \,
    \end{align*}
    where the second step follows from the fact that $\sqrt{2}$ is irrational, and the third step is a union bound.
    Expressing $\size{V_A}$ and $\size{V_B}$ as a sum of the indicator variables now yields
    \begin{align*}
        3\size{V_A} - 2\sqrt{2} \size{V_B} &= 3\sum_{i=1}^n \cdot\mathds{1}_A(X_i) - 2\sqrt{2} \sum_{i=1}^n \cdot\mathds{1}_B(X_i)\\ &= \sum_{i=1}^n \left(3\cdot\mathds{1}_A - 2\sqrt{2}\cdot\mathds{1}_B\right)(X_i) = \sum_{i=1}^n Y_i\ .
    \end{align*}
    Together, 
    \begin{multline*}
        \Pr\limits_{D \sim \mathbb{D}}\mspace{-3mu} \left( \CT_{G(D, w_{\delta})}(V_B) \text{ is a tangle that contains } V_S\right)\\
    \geq \Pr\limits_{D \sim \mathbb{D}}\mspace{-3mu}\left(\sum\limits_{i=1}^n Y_i \leq 0\mspace{-3mu}\right) + \Pr\limits_{D \sim \mathbb{D}}\left(\size{V_B} \geq 2\right) - 1 \ .
    \end{multline*}

Concerning the bounds on the probability that $\sum_{i=1}^n Y_i \leq 0$, we prove two bounds separately and then take the maximum. We first observe that $Y_1, \ldots, Y_n$ are independent random variables with $-2\sqrt{2} \leq Y_i \leq 3$ for all $i \in [n]$. Furthermore, for $\mean{Y} = \frac{1}{n}\sum_{i=1}^n Y_i$, we have
\begin{align*}
\EX\left(\mean{Y}\right) &= \frac{1}{n}\sum_{i=1}^n\EX\big((3 \cdot \mathds{1}_{A} - 2\sqrt{2} \cdot \mathds{1}_{B})(X_i)\big) = \frac{1}{n}\sum_{i=1}^n 3\nu_{\ell(i)}(A) - 2\sqrt{2}\nu_{\ell(i)}(B) \\&= 3\mean{\nu}(A) - 2\sqrt{2}\mean{\nu}(B) < 0.
\end{align*}
Now the first bound follows from \Cref{cor:hoeffding} with $t = -\EX(\mean{Y})$, since
\begin{align*}
    \Pr_{D \sim \DD}\left(\sum_{i=1}^n Y_i \leq 0\right) &\geq 1 -  \Pr_{D \sim \DD}(\mean{Y} - \EX(\mean{Y}) \\
    &\geq - \EX(\mean{Y})) > 1 - \left(-\frac{2n^2\EX(\mean{Y})^2}{\sum_{i=1}^n (3 + 2\sqrt{2})^2}\right) \\
    &=1 - \exp\left(-n \cdot \frac{2}{(2\sqrt{2} + 3)^2}\cdot(3\mean{\nu}(A) - 2 \sqrt{2} \mean{\nu}(B))^2\right)
\end{align*}
and the second bound follows from \Cref{thm:berry-esseen}, since
\begin{align*}
\Pr_{D \sim \DD}\left(\sum_{i=1}^n Y_i \leq 0\right) &= \Pr_{D \sim \DD}\left(\frac{\sum_{i=1}^n Y_i - \sum_{i=1}^n \EX(Y_i)}{\sqrt{\sum_{i=1}^n \Var(Y_i)}} \leq \frac{- \sum_{i=1}^n \EX(Y_i)}{\sqrt{\sum_{i=1}^n \Var(Y_i)}}\right) \\
&\geq \Phi_{0,1}\left(\frac{- \sum_{i=1}^n \EX(Y_i)}{\sqrt{\sum_{i=1}^n \Var(Y_i)}}\right) - \frac{C \cdot \sum_{i=1}^n \rho_i}{(\sum_{i=1}^n \Var(Y_i))^\frac{3}{2}}
\end{align*}
For the probability that $\size{V_B} \geq 2$, we can apply \Cref{lem:basic_computations_app}, Part \ref{enum:size_geq_two_app}. Altogether, we obtain the desired lower bound.

To prove the lower asymptotic bound of $1 - e^{-\Omega(n)}$, we first observe that
\begin{multline*}
\left(1 +  \sum_{k=1}^m  \frac{n_k \cdot \nu_k(B)}{1 - \nu_k(B)}\right) \cdot \prod_{k=1}^m \big(1 - \nu_k(B)\big)^{n_k}\\
= \left(1 + n\cdot \sum_{k=1}^m  \frac{r_k \cdot \nu_k(B)}{1 - \nu_k(B)}\right) \cdot \left(\prod_{k=1}^m \big(1 - \nu_k(B)\big)^{r_k}\right)^n\mspace{-4mu},
\end{multline*}
which shows that there are constants $c_1, c_2 >0$ and $c_3 \in (0,1)$ such that
\begin{align*}
    \Pr_{D \sim \mathbb{D}}&\big( \CT_{G(D, w_{\delta})}(V_B) \text{ is a tangle that contains } V_S\big) \\
    {}\geq{}& 1 - \exp\left(-n \cdot \frac{2 \cdot (3\mean{\nu}(A) - 2 \sqrt{2} \mean{\nu}(B)\big)^2}{(2\sqrt{2} + 3)^2}\right)\\
    	&- \left(1 +  \sum_{k=1}^m  \frac{n_k \cdot \nu_k(B)}{1 - \nu_k(B)}\right) \cdot \prod_{k=1}^m \big(1 - \nu_k(B)\big)^{n_k}\\
    {}={} &1 - e^{-c_1 \cdot n} - (1 + c_2\cdot n)\cdot c_3^n = 1 - e^{-c_1 \cdot n} - e^{\ln(1 + c_2\cdot n) + \ln(c_3) \cdot n} \ .
    \end{align*}
    Since $\ln(c_3) < 0$ and $\ln(1 + c_2\cdot n) = o(n)$, this completes the proof.
\end{proof}

Note that the bound derived from Berry-Esseen’s Inequality (\Cref{thm:berry-esseen}) only yields an asymptotic bound of $1 - O(\frac{1}{\sqrt{n}})$.
However, when $n$ is small, it often gives better results than the bound derived from  Hoeffding’s Inequality.
Using the union bound with the event $A_k$ being ``$\CT_{G(D, w_{\delta})}(V_{B_{\frac{\delta}{2}}[\vec{\mu}_k]})$ is a tangle that contains $V_S$'', we can apply Theorems \ref{thm:large_n_delta} and~\ref{thm:small_n_delta} to bound the probability that the tangles we consider are incomparable.

\begin{corollary}
Let $\delta > 0$, $d, m \in \NN_{>0}$ and let $(r_1, \vec{\mu}_1, \sigma_1), \ldots, (r_m, \vec{\mu}_m, \sigma_m)$ be as in \Cref{def:dataset}.  Suppose that there are Borel sets $S_{k_1, k_2}$ for all $\{k_1, k_2\} \in \binom{[m]}{2}$ with $S_{k_1, k_2} = S_{k_2,k_1}^c$.
If for each pair $(k_1,k_2)$, the conditions of \Cref{thm:large_n_delta} with $k = k_1$ and $S = S_{k_1, k_2}$ are fulfilled, then 
\[
\smash{\Pr_{D \sim \DD}}\big((\CT_{G(D, w_{\delta})}(V_{B_{\frac{\delta}{2}}[\vec{\mu}_k]}))_{k\in[m]} \text{ are pairwise incomparable tangles}\big) = 1 - O\left(\tfrac{1}{n}\right)
\]
for compatible $n$ and $\ell$ and $\mathbb{D} = \mathbb{D}[(r_k, \vec{\mu}_k, \sigma_k)_{k \in [m]},n,\ell]$. If all $k_1$ and $S_{k_1,k_2}$  meet the conditions of \Cref{thm:small_n_delta},  then 
\[
\smash{\Pr_{D \sim \DD}}\big((\CT_{G(D, w_{\delta})}(V_{B_{\frac{\delta}{2}}[\vec{\mu}_k]}))_{k\in[m]} \text{ are pairwise incomparable tangles}\big) = 1 - e^{-\Omega(n)} \ .
\]
\label{cor:inc_tangles_delta}
\end{corollary}

\subsection{Computational results}\label{subsec:comp}

So far, we have seen theorems of the form ``If some condition is met, there will be incomparable tangles with at least probability $P$''. However, it is not obvious to see for which parameters their conditions are fulfilled and which bounds they yield. In this section, we provide an analysis of the implications of our theorems for some explicit parameter sets, i.e., we study the significance of the theoretically obtained bounds.

We begin our analysis with the case of two Gaussians, starting with an in-depth study of the one-dimensional setting followed by the generalization to higher dimensions. Then, we analyse a case of more than two marginals by studying a mixture of three Gaussians whose means form an equilateral triangle. We explore reasonable choices for the Borel set $S$ that defines the candidate low-order cut such that the conditions for our probability bounds are met also for Gaussian mixtures where the means of the marginals are close to each other.
Since the structural properties of the data only depend on the ratio of the distances of the means and the standard deviation\footnote{The ratio $s^2=\frac{\lambda^2}{\sigma^2}$ of the smallest distance of means and the standard deviation is often called signal-to-noise ratio (SNR).}, we can normalize and restrict the following analysis to mixtures where all marginal distributions have standard deviation 1 or, in the case of a mixture with different standard deviations, we assume that at least one marginal has standard deviation 1.

We start with the \basecase\ of two one-dimensional Gaussians with means $0$ and $\lambda$ and equal standard deviations $\sigma$ and equally many points drawn from each.
In this setting, we need $\lambda>2\sigma$ in order to have more than one dense region in the data, because for $\lambda\leq 2\sigma$, the mean density $\bar{f}$ only has one maximum at $\frac{\lambda}{2}$.
As the low-order cut used to distinguish the tangles, we choose the cut at $\frac{\lambda}{2}$, so $S=\left(-\infty,\frac{\lambda}{2}\right]$. As shown in \Cref{fig:small_n_1dim_delta_mu}, \Cref{thm:small_n_delta} yields good lower bounds already for small data sets.

This setting can be generalized in three ways.
First, we vary the fraction $r$ of data points drawn from one marginal distribution, i.e., we assume $rn$ data points are drawn from the first marginal and $(1-r)n$ from the second marginal. Here, the considered cut is at the local minimum of the average density. The implied probability bounds and bounds on the distance of means for this case are shown in \Cref{fig:small_n_1dim_delta_frac_largeN,fig:1dim_2distr_expectation_mixing_delta}.
Next, we vary the ratio $\alpha$ of the standard deviations, see \Cref{fig:1dim_2distr_expectation_mixing_sigma}. Again, we separate the regions at the local minimum of the mean distribution. In both of these first two generalizations, we can compute the smallest distance of the means where the conditions of \Cref{cor:inc_tangles_delta} can be met via the following lemma.

\begin{lemma}\label{thm:limit_delta_1dim}
    Let $\lambda, \sigma, \alpha > \mspace{-2mu} 0$ and $r_1\mspace{-2mu}\in\mspace{-2mu}(0,1)$. Let $m \coloneqq 2$, $\mu_1 \coloneqq 0$, $\mu_2 \coloneqq \lambda$, $\sigma_1 \coloneqq \sigma$, $\sigma_2 \coloneqq \alpha \cdot \sigma$, $r_2 \coloneqq 1-r_1$. Then with $c\coloneqq\operatorname{argmin}_{0\leq x \leq\lambda} \bar{f}(x)$ and \Cref{not:density}, the following hold for compatible $n$ and $\ell$ and $\mathbb{D} = \mathbb{D}[((r, 0, \sigma),(1-r,\lambda,\alpha\sigma)),n,\ell]$.
    \begin{enumerate}
	    \item If $\frac{2}{3} \cdot \min \{\bar{f}(0), \bar{f}(\lambda)\} > \bar{f}(c)$, there is a $\delta > 0$ such that
	\[
	\smash{\Pr_{D \sim \mathbb{D}}}\left(\begin{matrix} \CT_{G(D, w_{\delta})}(
	V_{B_{\delta/2}[0]}) \text{ and } \CT_{G(D, w_{\delta})}(V_{B_{\delta/2}[\lambda]})\\ \text{ are incomparable tangles}
	\end{matrix}\right) = 1 - O\left(\tfrac{1}{n}\right) \  .
	\]
	\item If $\frac{\sqrt{2}}{3} \cdot \min \{\bar{f}(0), \bar{f}(\lambda)\} > \bar{f}(c)$, there is a $\delta > 0$ such that
	\[
	\smash{\Pr_{D \sim \mathbb{D}}}\left( \begin{matrix}
	\CT_{G(D, w_{\delta})}(V_{B_{\delta/2}[0]}) \text{ and } \CT_{G(D, w_{\delta})}(V_{B_{\delta/2}[\lambda]})\\ \text{ are incomparable tangles}
	\end{matrix}\right) = 1 - e^{-\Omega(n)} \ .
	\]
\end{enumerate}
\end{lemma}

\noindent
\begin{figure}
	\begin{subfigure}[t]{0.49\textwidth}
		\textbf{(a)}
		\vtop{
			\vskip-1ex
			\hbox{
				\includegraphics[width=0.88\textwidth]{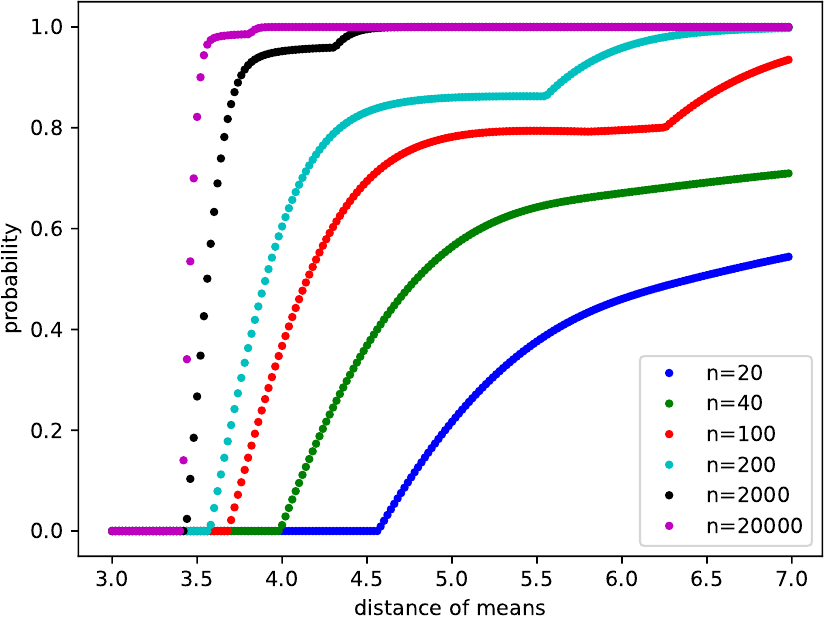}
			}
		}
		\captionsetup{width=6cm}\phantomcaption{~}
		\label{fig:small_n_1dim_delta_mu}
	\end{subfigure}~
	\begin{subfigure}[t]{0.49\textwidth}
		\textbf{(b)}
		\vtop{
			\vskip-1ex
			\hbox{
				\includegraphics[width=0.88\textwidth]{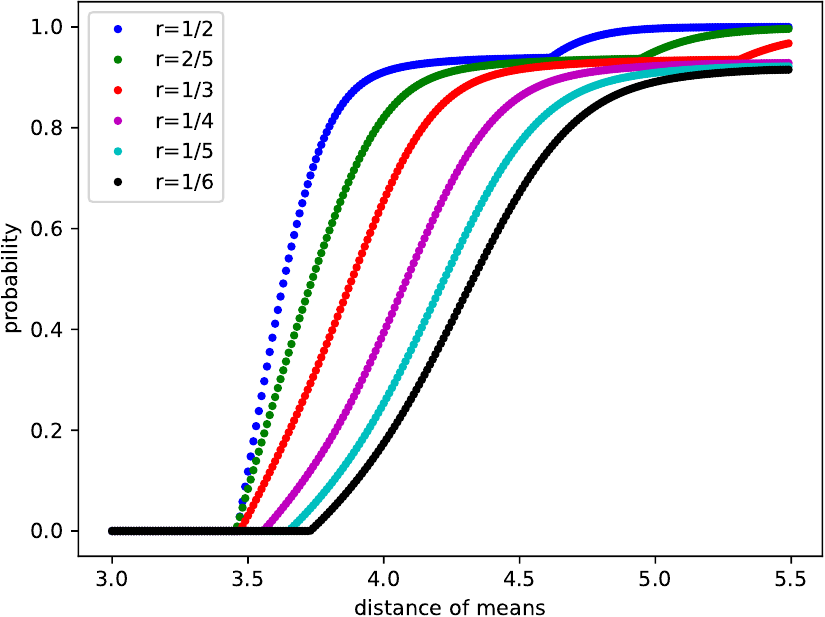}
			}
		}
		\captionsetup{width=6cm}\phantomcaption{~}
		\label{fig:small_n_1dim_delta_frac_largeN}
	\end{subfigure}\\
	\begin{subfigure}[t]{0.49\textwidth}
		\textbf{(c)}
		\vtop{
			\vskip-1ex
			\hbox{
				\includegraphics[width=0.88\textwidth]{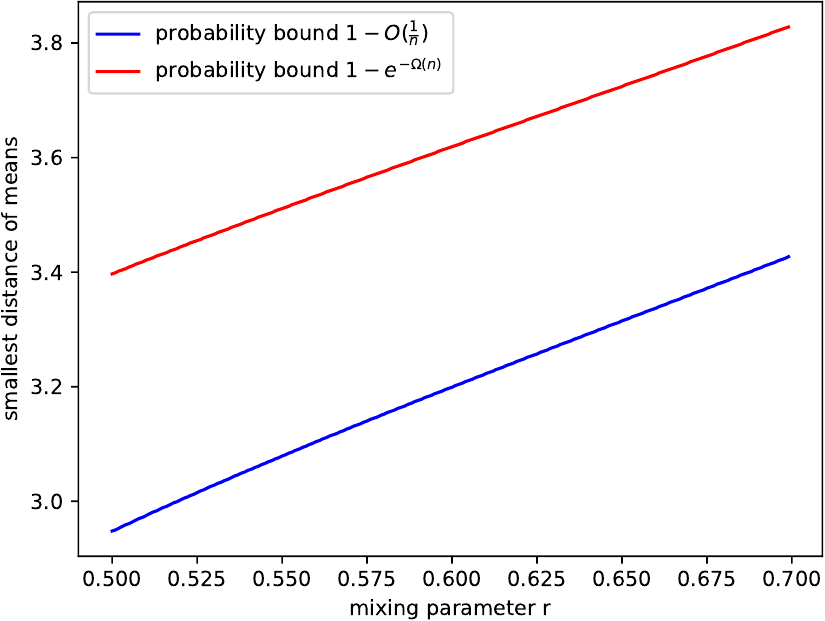}
			}
		}
		\captionsetup{width=6cm}\phantomcaption{~}
		\label{fig:1dim_2distr_expectation_mixing_delta}
	\end{subfigure}~
	\begin{subfigure}[t]{0.49\textwidth}
		\textbf{(d)}
		\vtop{
			\vskip-1ex
			\hbox{
				\includegraphics[width=0.88\textwidth]{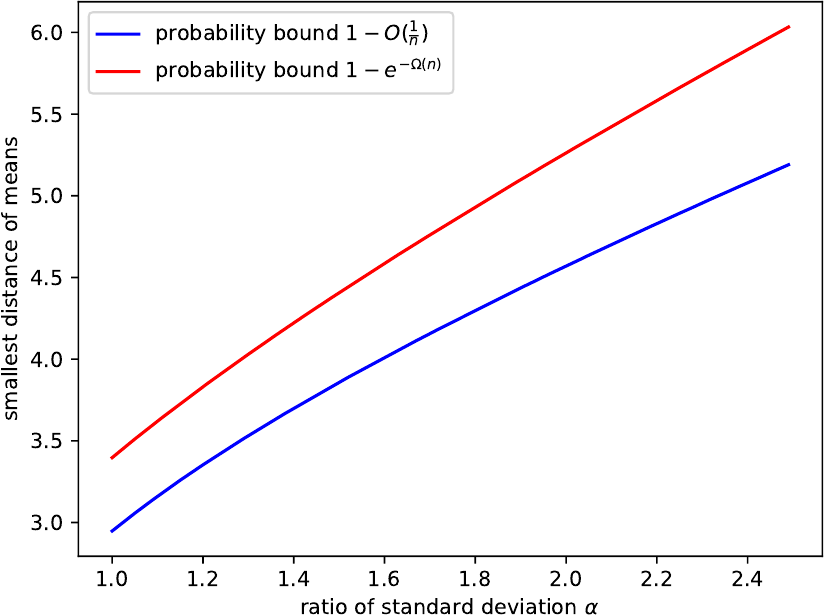}
			}
		}
		\captionsetup{width=6cm}\phantomcaption{~}
		\label{fig:1dim_2distr_expectation_mixing_sigma}
	\end{subfigure}
	
	\caption{Our computational results in the one-dimensional case.\newline
		We study the applicability of \Cref{thm:small_n_delta}, where we choose $\delta$ to optimize the probability bound. In the one-dimensional case, we assume that the means of the marginal distributions have distance $\lambda$ and the standard deviations are chosen to be $1$ and $\alpha$. From the first marginal distribution we draw $rn$ many data points, from the second one, we draw $(1-r)n$.
		\textbf{(a)} shows the probability bound dependent on $\lambda$ for different $n$ in the \basecase. 
		In
		\textbf{(b)}, we vary the mixing parameters $r$ and plot the probability bound dependent on $\lambda$ with $n=900$.
		\newline
		\textbf{(c)}, \textbf{(d)} show plots of the smallest mean distance $\lambda$ such that the conditions of \Cref{thm:limit_delta_1dim} are met.
		In \textbf{(c)} we vary the mixing parameter $r$ and in \textbf{(d)}, we vary the ratio $\alpha$ between the standard deviations. Fixing $r=1/2$ and $\alpha=1$, we obtain the bounds $\lambda>2.948$ and $\lambda>3.397$, respectively.\newline
		Note that at some point the bound corresponding to Hoeffding's Inequality (\Cref{cor:hoeffding}) becomes larger than the bounds relating to \Cref{thm:berry-esseen}.
	}
	\label{fig:results_delta}
\end{figure}

\begin{proof}[Proof sketch]
We observe for $x\in\RR$ and $\alpha >0$ that $\lim_{\delta \to 0} \frac{1}{\delta}\mean{\nu}([x - \alpha\cdot \delta, x + \alpha\cdot \delta]) = 2\alpha\cdot \mean{f}(x)$ and $\lim_{\delta \to 0} \frac{1}{\delta^2}\int_{c-\delta}^c\int_c^{x+\delta} \bar{f}(x) \bar{f}(y) \, dy \, dx = \frac{1}{2} \mean{f}(c)^2$. Thus, by the assumptions, the conditions of Theorems \ref{thm:large_n_delta} and \ref{thm:small_n_delta}, respectively, hold in the limit for $\delta \to 0$. Since the conditions of the theorems are strict inequalities, there must be a $\delta > 0$ that satisfies them.
\end{proof}

Lastly, we vary the dimension of the data.
For dimensions two and three, \Cref{thm:small_n_delta} yields probability bounds similar to the one-dimensional case, see \Cref{fig:2dim_2distr_delta} and  \Cref{fig:3dim_2distr_delta}.
Since the exact values for $\mean{\nu}(B_{\frac{\delta}{2}}[\vec{\mu_i}])$ become difficult to compute in high dimensions, we approximate them by $\mean{\nu}(C)$, where $C$ is the largest hypercube contained in the hyperball as pictured in \Cref{fig:2dim_2distr_datadistr}.
We then compute the smallest distance of means $\lambda$ for which the conditions of \Cref{thm:small_n_delta} are met, the results are shown in \Cref{fig:hdim_2distr_delta}. In accordance with our intuition, they indicate that this approach works well for low-dimensional data.

\begin{figure}
	\begin{minipage}[t]{0.3\textwidth}
		\strut\vspace*{-\baselineskip}\newline
		\begin{subfigure}[t]{0.7\textwidth}
			\begin{tikzpicture}
			\node[anchor=north east] at (0, 0) {\llap{\textbf{(a)}}};
			\node[anchor=north west] at (0, 0) {
				\begin{tikzpicture}[scale=0.23]
				\fill[fill=red!40] (-3.5,0) circle (1cm);
				\fill[fill=red!40] (3.5,0) circle (1cm);
				\fill[fill=red!85!black] (-4.2071,-0.7071) rectangle (-2.7929,0.7071);
				\fill[fill=red!85!black] (4.2071,0.7071) rectangle (2.7929,-0.7071);
				\fill[fill=blue!40] (-2,-5.5) rectangle (2,5.5);
				\draw[thick, draw=blue] (0,-5.5) -- (0,5.5);
				\end{tikzpicture}
			};
		\end{tikzpicture}
		\phantomcaption{~}
		\label{fig:2dim_2distr_datadistr}
	\end{subfigure}~\hspace{0.5cm}
\end{minipage}
\hfill
\begin{minipage}[t]{0.3\textwidth}
	\strut\vspace*{-\baselineskip}\newline
	\begin{subfigure}[t]{0.7\textwidth}
		\begin{tikzpicture}
		\node[anchor=north east] at (0, 0) {\llap{\textbf{(b)}}};
		\node[anchor=north west] at (0, 0) {
			\begin{tikzpicture}[scale=0.23]
			\fill[fill=red!40] (0,0) circle (1cm);
			\fill[fill=red!40] (7,0) circle (1cm);
			\fill[fill=red!40] (3.5, -6.0622) circle (1cm);
			\fill[fill=blue!40] (3.5,-2.0207) circle (2cm);
			\fill[fill=blue!40] (5.5,3) -- (1.5,3) -- (1.5,-0.8660) -- (-1.2,-2.4249) -- (-1.2,-6.8) -- (4.5,-3.7528) --(5.5,-2.0207) -- (5.5,1.2);
			\draw[thick, draw=blue] (3.5,3) -- (3.5,-2.0207) -- (-1.2,-4.6188);
			\end{tikzpicture}
		};
	\end{tikzpicture}
	\phantomcaption{~}
	\label{fig:2dim_3distr_datadistr}
\end{subfigure}~\hspace{0.5cm}
\end{minipage}
\hfill
\begin{minipage}[t]{0.3\textwidth}
\strut\vspace*{-\baselineskip}\newline
\begin{subfigure}[t]{\textwidth}
	\begin{tikzpicture}
	\node[anchor=north east] at (0, 0) {\llap{\textbf{(c)}}};
	\node[anchor=north west] at (0, 0) {
		\begin{tikzpicture}[scale=0.23]
		\clip(-4.5,-7.1) rectangle (8,4.5);
		\fill[fill=blue!40, rotate=15] (-5,-5) rectangle (5,5);
		\fill[fill=white, rotate=15] (-1,-1) rectangle (1,1);
		\draw[thick, draw=blue, rotate=15] (-3,-3) rectangle (3,3);
		\fill[fill=red!40] (0,0) circle (1cm);
		\fill[fill=red!40] (7,0) circle (1cm);
		\fill[fill=red!40] (3.5, -6.0622) circle (1cm);
		\end{tikzpicture}
	};
\end{tikzpicture}
\phantomcaption{~}
\label{fig:2dim_3distr_datadistr_square}
\end{subfigure}
\end{minipage}
\caption{A schematic image of the higher-dimensional data distribution models. We take marginals with equal $\sigma$ and draw equally many points from each distribution. The red circles represent the means and the associated hyperballs. In blue, we see the low-order cut, in light blue the area of points that possibly contribute to the order of the cut. In \textbf{(a)} we have two marginal distributions. The low-order cut is $S=\{(x_1,\ldots,x_d)\mid x_1\leq \frac{\lambda}{2}\}$. The approximation of the hyperball  as a hypercube is shown in dark red. \textbf{(b)} and \textbf{(c)} show 3 distributions whose means are positioned on an equilateral triangle. The cut along the Voronoi cell is shown in \textbf{(b)}, in \textbf{(c)} we see the cut along a cube centered at one of the means.}
\label{fig:2dim_datadistr}
\end{figure}
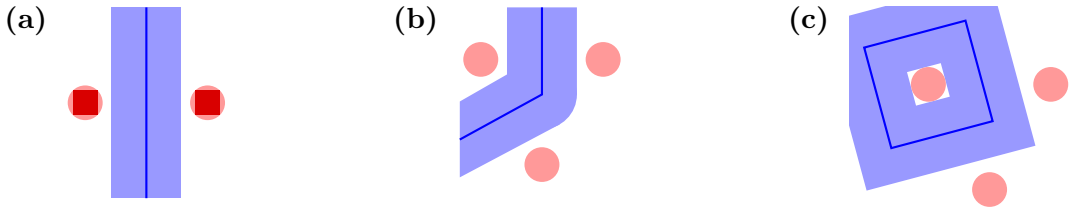

When we consider more than two Gaussians, the positions of the means become crucial. If the means are aligned (as is always the case in one dimension), they are naturally ordered and a cut separating consecutive means also separates all smaller from all larger means.
Also, all distributions around non-adjacent means only contribute insignificantly to the cut.
Hence, the analysis is essentially like the one for a mixture of just two marginal distributions.

However, if the means are not aligned, the best choice for the Borel set to determine the low-order cut is no longer obvious.
We demonstrate this with the example of three two-dimensional Gaussians whose means form an equilateral triangle. Here, there are several possible choices to define the cuts.
Any straight line separating one of the means from the others will be too close to one dense region, resulting in a high-order cut.
For this reason, we investigate two alternative cuts: one along the \emph{Voronoi cell}, that is $S_i=\{\vec{x}\mid \lVert \vec{x}-\vec{\mu_i}\rVert\leq\lVert\vec{x}-\vec{\mu_k}\rVert \text{ for all } k\neq i\}$, and one along a square centered at the mean, see \Cref{fig:2dim_3distr_datadistr} and \Cref{fig:2dim_3distr_datadistr_square}.
The cut along the Voronoi cell works well for \Cref{thm:large_n_delta} and yields the existence of incomparable tangles asymptotically almost surely as soon as $\lambda>4.1$.
Notably, the square approach yields much better results in the application of \Cref{thm:small_n_delta}, as shown in \Cref{fig:2dim_3distr_delta}.

\begin{figure}
	\begin{subfigure}[t]{0.49\textwidth}
		\textbf{(a)}
		\vtop{
			\vskip-1ex
			\hbox{
				\includegraphics[width=0.89\textwidth]{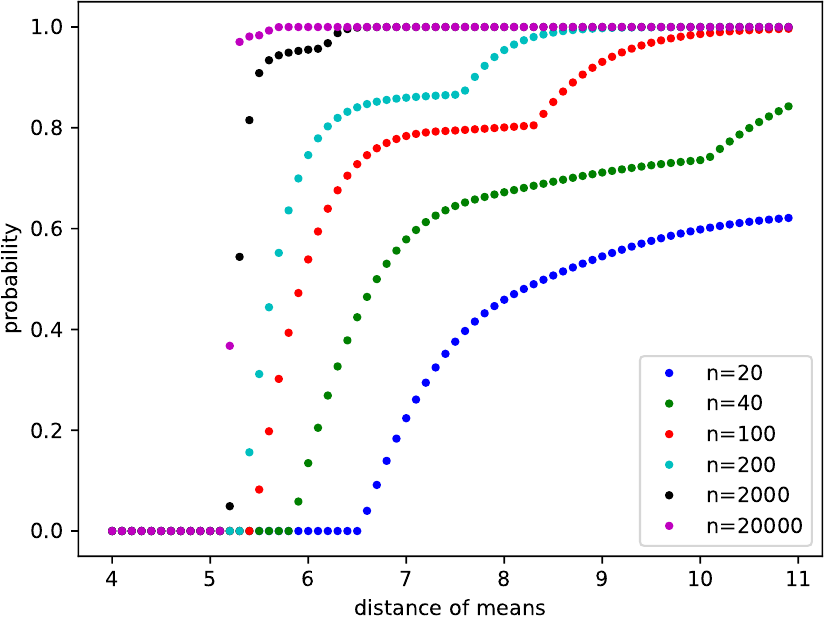}
			}
		}
		\captionsetup{width=6cm}\phantomcaption{~}
		\label{fig:2dim_2distr_delta}
	\end{subfigure}~
	\begin{subfigure}[t]{0.49\textwidth}
		\textbf{(b)}
		\vtop{
			\vskip-1ex
			\hbox{
				\includegraphics[width=0.89\textwidth]{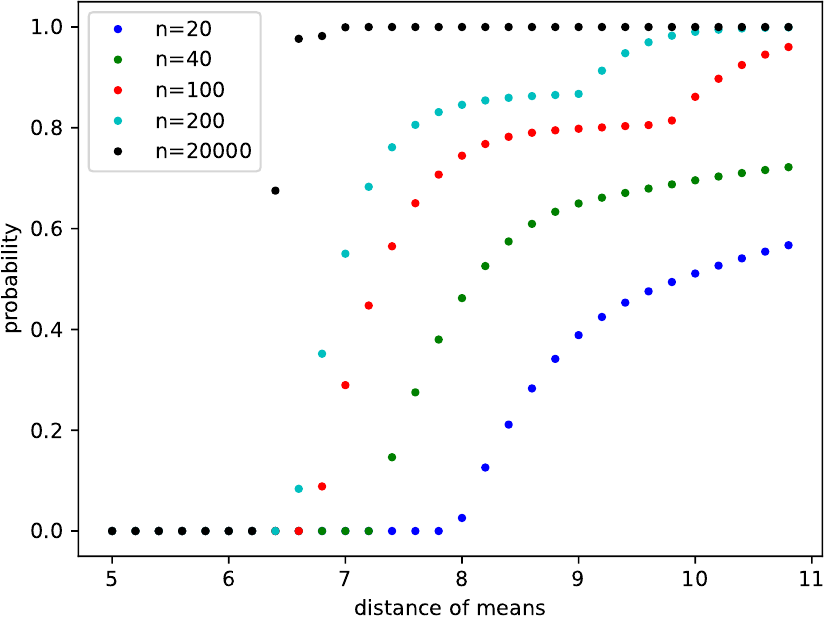}
			}
		}
		\captionsetup{width=6cm}\phantomcaption{~}
		\label{fig:3dim_2distr_delta}
	\end{subfigure}
	
	\begin{subfigure}[t]{0.49\textwidth}
		\textbf{(c)}
		\vtop{
			\vskip-1ex
			\hbox{
				\includegraphics[width=0.88\textwidth]{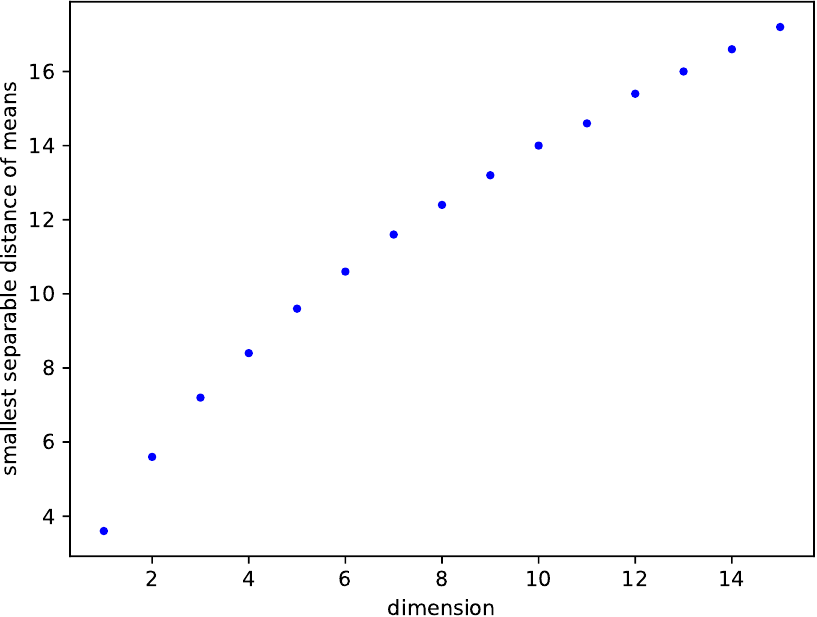}
			}
		}
		\captionsetup{width=6cm}\phantomcaption{~}
		\label{fig:hdim_2distr_delta}
	\end{subfigure}~
	\begin{subfigure}[t]{0.49\textwidth}
		\textbf{(d)}
		\vtop{
			\vskip-1ex
			\hbox{
				\includegraphics[width=0.88\textwidth]{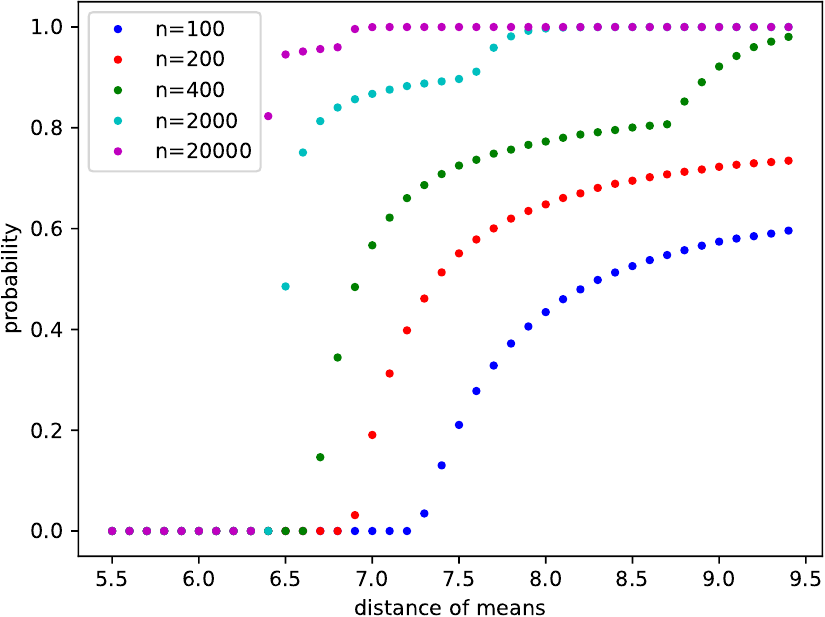}
			}
		}
		\captionsetup{width=6cm}\phantomcaption{~}
		\label{fig:2dim_3distr_delta}
	\end{subfigure}
	\caption{Our computations in higher dimensions.\newline
		First we consider a mixture of two types of distributions with mean distance $\lambda$ (see \Cref{fig:2dim_2distr_datadistr}).
		We plot the probability bounds from \Cref{thm:small_n_delta} dependent on $\lambda$, where the dimension is two in \textbf{(a)} and three in \textbf{(b)}.\newline
		In \textbf{(c)} we plot the smallest $\lambda$ such that the conditions of \Cref{thm:large_n_delta} are met and we get incomparable tangles a.a.s.\ dependent on the dimension, using the described approximation.\newline
		In \textbf{(d)} we take three distributions whose means form an equilateral triangle with side length $\lambda$ (see \Cref{fig:2dim_3distr_datadistr_square}).
		We plot the probability bound from \Cref{thm:small_n_delta} for dimension 2 against the side length $\lambda$.
		The size of the hypercube is chosen to maximize the resulting lower bound on the probability.\newline
		Again we see that at some point the bound corresponding to Hoeffding's Inequality (\Cref{cor:hoeffding}) becomes larger than the bounds relating to \Cref{thm:berry-esseen}.
	}
	\label{fig:multidim_results}
\end{figure}

\section{Results for the fully connected graph}
\label{sec:complete}
We also investigate to which extent the theorems derived for the \deltagraph\ in \Cref{sec:probability_delta} translate to the \completegraph. As the weight function, we take $\hat{w}_{\sigma}(\vec{x},\vec{y}) \coloneqq \exp\left(\frac{- \lVert \vec{x} - \vec{y} \rVert^2}{2\sigma^2}\right)$, where $\sigma$ is the standard deviation of each of the marginal distributions. Like before, we start by identifying incomparable tangles induced by hyperballs around the means. Using \Cref{lem:clique-tangle}, as soon as the hyperball contains at least two data points, there is a tangle for this region. The order of the tangle depends on the smallest weight among the contained edges. Let the diameter of the hyperballs be $\Delta$. We observe that the number of vertices within a hyperball grows with $\Delta$, whereas the minimum edge weight decreases. This yields an interesting trade-off for obtaining a good tangle order. Note that, in contrast to the \deltagraph, the choice of $\Delta$ does not influence the cut order.

Again, we first deduce a general bound on the probability that, for a given separation, the $\frac{\Delta}{2}$-ball around the mean of a distribution induces a tangle with a prescribed orientation of the separation.

\begin{theorem}	\label{thm:large_n_weight}
    Let $d, m \in \NN_{>0}$. Fix $\sigma > 0$ and let $(r_1, \vec{\mu}_1), \ldots, (r_m, \vec{\mu}_m)$ as in \Cref{def:dataset}. Choose $\Delta > 0$, $k \in [m]$ and define $B \coloneqq B_{\frac{\Delta}{2}}[\vec{\mu_k}]$. Let $S \subseteq \RR^d$ be a Borel set. If 
    \begin{equation*}
        \mean{\nu}(B \cap S) > \tfrac{1}{2} \mean{\nu} (B) 
    \text{ and } 
        \tfrac{2}{9} e^{-\frac{\Delta^2}{2\sigma^2}}\cdot \mean{\nu}(B)^2 > \smash{\int_{S}} \smash{\int_{\compl{S}}} \hat{w}_{c}(\vec{x},\vec{y}) \mean{f}(\vec{x}) \mean{f}(\vec{y}) \, d\vec{y} \, d\vec{x} \, ,
    \end{equation*}
    then for compatible $n$ and $\ell$ and $\mathbb{D} = \mathbb{D}[(r_k, \vec{\mu}_k, \sigma)_{k \in [m]},n,\ell]$, it holds that
    \[
    \smash{\Pr_{D \sim \mathbb{D}}}\left( \CT_{G(D, \hat{w}_{\sigma})}(V_B) \text{ is a tangle that contains } V_S
    \right) = 1 - O\left(\tfrac{1}{n}\right) \ .
    \]
\end{theorem}

\begin{proof}
    We follow the proof of \Cref{thm:large_n_delta}, but in contrast to the \deltagraph, $w_{V_B}$ is a random variable for the \completegraph. We observe however, that $w_{V_B} \geq e^{\frac{-\Delta^2}{2\sigma^2}}$ as two points in $B$ have distance at most $\Delta$. By comparing $\kappa_{G(D, \hat{w}_{\sigma})}(V_S)$ with the lower bound on $w_{V_B}$ instead of the actual value, the rest of the proof is analogous to the one of \Cref{thm:large_n_delta}.
\end{proof}

As in \Cref{sec:delta}, we also want conditions to ensure a faster convergence of the probabilities. Unfortunately, \Cref{lem:bound_kappa} does not easily translate to the \completegraph\ in general. Restricting ourselves to the one-dimensional setting, we find the following suitable inequality.

\begin{lemma}\label{lem:bound_kappa_complete}
Suppose $\sigma > 0$, $n \in \NN_{>0}$, $D = (x_1, \ldots, x_n)\in\RR^{n}$ and $c \in \RR$. Then
\[\kappa_{G(D, \hat{w}_{\sigma})}\big(V_{(-\infty,c]}\big) \leq \frac{1}{4}\left(\sum_{i=1}^n e^{\frac{-(x_i-c)^2}{2\sigma^2}}\right)^2 \ .\]
\end{lemma}
\begin{proof}
    Let $\{i,j\}$ be an edge that contributes to the cut, say $i \in V_{(-\infty,c]}$ and $j \in V_{(c,\infty)}$. Then $x_i \leq c < x_j$ and hence $x_j - c$ and $c - x_i$ are non-negative. This yields $(x_j - x_i)^2 = (x_j - c + c - x_i)^2 \geq (x_j - c)^2 + (c - x_i)^2$ and hence 
    \[\hat{w}_{\sigma}(x_i, x_j) = e^{\frac{-(x_i - x_j)^2}{2\sigma^2}} \leq e^{\frac{-(x_i - c)^2 - (c - x_j)^2}{2\sigma^2}} = e^{\frac{-(x_i - c)^2}{2\sigma^2}} \cdot e^{\frac{-(x_j - c)^2}{2\sigma^2}}.\]
    Using this insight and the inequality $a \cdot b \leq \frac{1}{4}(a+b)^2$ for reals $a$ and $b$, we conclude
    \begin{align*}
    \kappa_{G(D, \hat{w}_{\sigma})}\big(V_{(-\infty,c]}\big) &= \sum_{i \colon x_i \leq c} \ \sum_{j \colon x_j > c}  e^{\frac{-(x_i - x_j)^2}{2\sigma^2}} \leq \sum_{i \colon x_i \leq c} \ \sum_{j \colon x_j > c}  e^{\frac{-(x_i - c)^2}{2\sigma^2}} \cdot e^{\frac{-(x_j - c)^2}{2\sigma^2}} \\
    &= \left(\sum_{i \colon x_i \leq c} e^{\frac{-(x_i - c)^2}{2\sigma^2}}\right) \cdot \left(\sum_{j \colon x_j > c} e^{\frac{-(x_j - c)^2}{2\sigma^2}}\right) \\
    &\leq \frac{1}{4} \left(\sum_{i \colon x_i \leq c} e^{\frac{-(x_i - c)^2}{2\sigma^2}} + \sum_{j \colon x_j > c} e^{\frac{-(x_j - c)^2}{2\sigma^2}}\right)^2\\
    &= \frac{1}{4} \left(\sum_{i=1}^n e^{\frac{-(x_i - c)^2}{2\sigma^2}} \right)^2
    \end{align*}
    as desired.
\end{proof}

We are now able to express the question whether a tangle orients a cut using a sum of independent random variables. This enables us to apply Hoeffding's inequality \Cref{cor:hoeffding} to get the desired stronger convergence.

\begin{theorem}\label{thm:small_n_weight}
	Let $m \in \NN_{>0}$. Fix $\sigma > 0$ and let $\mu_1,\ldots,\mu_k \in \RR$ and $r_1, \ldots, r_k \in \RR_{>0}$ with $r_1 + \ldots+ r_k = 1$. Choose $\Delta >0$, $k \in [m]$ and define $I \coloneqq [\mu_k - \frac{\Delta}{2},\mu_k + \frac{\Delta}{2}]$. Let $S$ be either $(-\infty,c]$ or $(c, \infty)$ for some $c \in \RR$ such that $I \subseteq S$.
	If 
	\[
	   \frac{\sqrt{2}}{3} \cdot e^{\frac{-\Delta^2}{4\sigma^2}} \cdot \mean{\nu}(I) > \sum_{k=1}^m \frac{r_k}{2\sqrt{2}} \cdot e^{\frac{-(\mu_k - c)^2}{4\sigma^2}} \ ,
	\]
	then, for all compatible $n$ and $\ell$ and $\mathbb{D} = \mathbb{D}[(r_k, \mu_k, \sigma)_{k \in [m]},n,\ell]$,
    \begin{align*}
    &\Pr_{D \sim \mathbb{D}}\left( \CT_{G(D, \hat{w}_{\sigma})}(V_I) \text{ is a tangle that contains } V_S
    \right) \\ 
    &{}\geq{} \  1 - \exp\left(
    -n \cdot 
    \frac{
    2 \cdot \left(\frac{\sqrt{2}}{3} \cdot e^{\frac{-\Delta^2}{4\sigma^2}} \cdot \mean{\nu}(I) - \sum_{k=1}^m \frac{r_k}{2\sqrt{2}} \cdot e^{\frac{-(\mu_k - c)^2}{4\sigma^2}}\right)^2
    }
    {
    \left(\frac{1}{2} + \frac{\sqrt{2}}{3} \cdot e^{\frac{-\Delta^2}{4\sigma^2}}  \right)^2
    } \right)\\
    &\phantom{{}\geq{} \ }- \left(1 +  \sum_{k=1}^m  \frac{n_k \cdot \nu_k(I)}{1 - \nu_k(I)}\right) \cdot \prod_{k=1}^m \big(1 - \nu_k(I)\big)^{n_k} \ .
    \end{align*}
    In particular,
    \[\Pr_{D \sim \mathbb{D}}\left( \CT_{G(D, \hat{w}_{\sigma})}(V_I) \text{ is a tangle that contains } V_S
    \right) = 1 - e^{-\Omega(n)} \ .\]
\end{theorem}
\begin{proof} For better readability, let $G \coloneqq G(D, \hat{w}_{\sigma})$.
As noticed before, by \Cref{lem:clique-tangle}, $\CT_G(V_I)$ is a tangle as soon as $\size{V_I} \geq 2$. Therefore, we can first use a union bound to obtain
\begin{multline*}
\Pr_{D \sim \mathbb{D}}\left( \CT_{G(D, \hat{w}_{\sigma})}(V_I) \text{ is a tangle that contains } V_S\right)\\
\geq \Pr_{D \sim \mathbb{D}}\big(V_S \in \CT_G(V_I)\big) + \Pr_{D \sim \mathbb{D}}\big(\size{V_i} \geq 2\big) - 1.
\end{multline*}
Since $V_I \subseteq V_S$ by assumption, $V_S$ is a member of $\CT_G(V_I)$ when $\frac{2}{9} w_{V_I} \size{V_I}^2 > \kappa_G(V_S)$. We observe $w_{V_I} \geq e^{\frac{-\Delta^2}{2\sigma^2}}$ as two points in $I$ have distance at most $\Delta$. Setting
\[Y_i \coloneqq \frac{1}{2}e^{\frac{-(X_i - c)^2}{2\sigma^2}} - \frac{\sqrt{2}}{3}e^{\frac{-\Delta^2}{4\sigma^2}}\cdot\mathds{1}_{I}(X_i) \ ,\]
we can use \Cref{lem:bound_kappa_complete} to obtain
\begin{align*}
    \Pr(V_S \in \CT_G(V_I))
    &=\Pr\left(\frac{2}{9} w_{V_I} \size{V_I}^2 > \kappa_G(V_S)\right) \\
    &\geq \Pr\left(\frac{2}{9} e^{\frac{-\Delta^2}{2\sigma^2}}\cdot\size{V_I}^2 > \frac{1}{4}\left(\sum_{i=1}^n e^{\frac{-(X_i - c)^2}{2\sigma^2}}\right)^2\right)\\
    &=\Pr\left( \frac{\sqrt{2}}{3} e^{\frac{-\Delta^2}{4\sigma^2}}\cdot\size{V_I} > \frac{1}{2}\sum_{i=1}^n e^{\frac{-(X_i - c)^2}{2\sigma^2}}\right)\\
    &=\Pr\left(\sum_{i=1}^n \frac{1}{2}\cdot e^{\frac{-(X_i - c)^2}{2\sigma^2}} - \frac{\sqrt{2}}{3} e^{\frac{-\Delta^2}{4\sigma^2}}\cdot \mathds{1}_I(X_i) < 0\right)\\
    &= \Pr\left( \sum_{i=1}^n Y_i < 0 \right) \ .
\end{align*}
Since the $Y_i$ are independent and for all $i \in [n]$, it holds that $\frac{\sqrt{2}}{3} e^{\frac{-\Delta^2}{4\sigma^2}} \leq Y_i \leq \frac{1}{2}$, we can apply \Cref{cor:hoeffding}. To obtain the expected value, we then compute
\[
\EX\left(e^{\frac{-(X_i - c)^2}{2\sigma^2}}\right) = \int_{\RR} e^{\frac{-(x - c)^2}{2\sigma^2}} \cdot e^{\frac{-(x - \mu_{\ell(i)})^2}{2\sigma^2}} \, dx =  \frac{1}{\sqrt{2}}\cdot e^{\frac{-(c - \mu_{\ell(i)})^2}{4\sigma^2}} \ .
\]
This yields
\[\EX(Y_i) = \frac{1}{2\sqrt{2}}\cdot e^{\frac{-(c - \mu_{\ell(i)})^2}{4\sigma^2}} - \frac{\sqrt{2}}{3} e^{\frac{-\Delta^2}{4\sigma^2}}\cdot \nu_{\ell(i)}(X_i) \ .\]
Now \Cref{cor:hoeffding} together with \Cref{lem:basic_computations_app}, Part \ref{enum:size_geq_two_app}, gives the desired bound.
\end{proof}

Figure \ref{fig:7} shows our computational results for the \completegraph\ in the \basecase\ of two one-dimensional Gaussian distributions with means $0$ and $\lambda$, equal standard deviation $\sigma$, and equally many points drawn from each marginal.
First we compute which distances of means $\lambda$ are separable with \Cref{thm:large_n_weight}, depending on the interval width $\Delta$, and find that we must have $\lambda>4.27$.

The preconditions of Theorems \ref{thm:large_n_delta} and \ref{thm:large_n_weight} represent the condition that, in expectation, the order of tangle must be larger than the edge connectivity of $V_S$.  Since every subset of data points induces a clique in the fully connected graph, we bound the size of the interval around the mean that we consider as a tangle-inducing clique in Lemma \ref{lem:clique-tangle}. Unlike in the case of the $\delta$-neighborhood graph, this parameter does not influence the value of the cut, but only influences the order of the tangle induced by the clique. As it turns out, picking a too narrow interval yields too few points in the clique and a too wide interval yields a too weak lower bound on the edge weight of each edge: in both cases, we only obtain a tangle of low order via Lemma \ref{lem:clique-tangle}. Figure \ref{fig:1dim_2distr_expectation_weight} shows the tradeoff between the interval width and the smallest separable distance of means.

Then we compute the lower bounds on the probability due to \Cref{thm:small_n_weight}.
We see that the distances are higher and the probabilities lower as in the same setting for the \deltagraph. 

\begin{figure}
	\begin{subfigure}[t]{0.49\textwidth}
	(a)
		\vtop{
            \vskip-1ex
            \hbox{
                \includegraphics[width=0.89\textwidth]{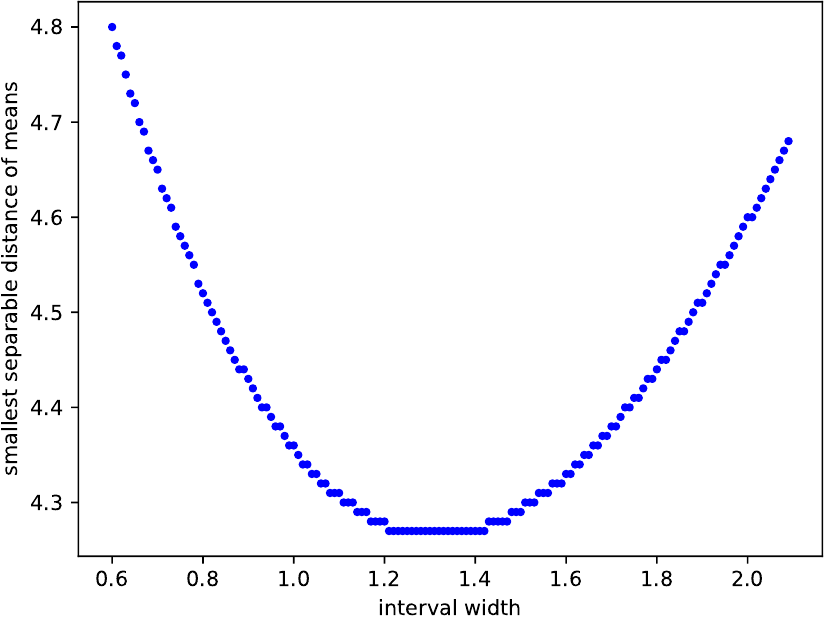}
                }
        }
		\phantomcaption{~}
	\label{fig:1dim_2distr_expectation_weight}
    \end{subfigure}~
	\begin{subfigure}[t]{0.49\textwidth}
	(a)
		\vtop{
            \vskip-1ex
            \hbox{
                \includegraphics[width=0.89\textwidth]{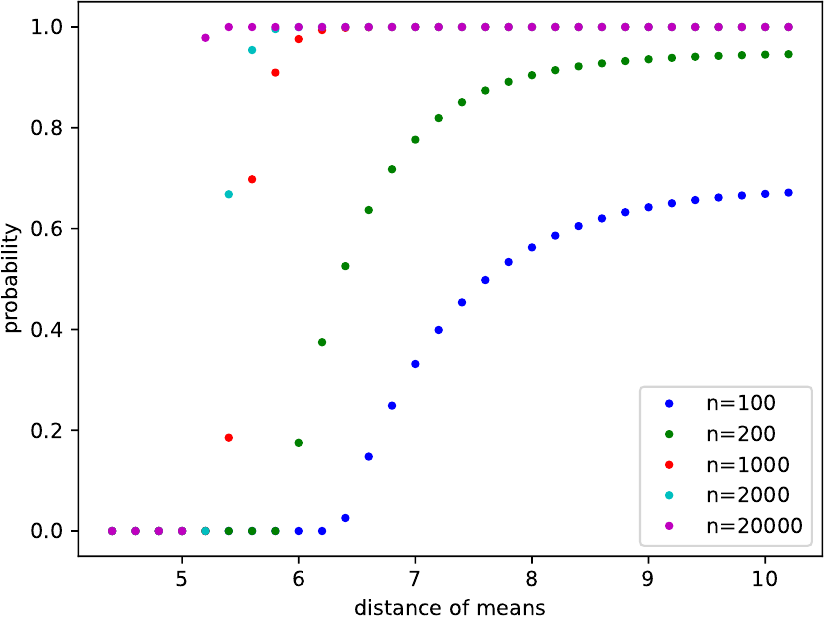}
                }
        }
		\phantomcaption{~}
	\label{fig:small_n_1dim_complete}
	\end{subfigure}
	\caption{Computational results for the \completegraph. The smallest distance of means $\lambda$ such that \Cref{thm:large_n_weight} can be applied, dependent on the width $\Delta$ of the interval used to define the tangles is shown in \textbf{(a)}. Plot \textbf{(b)} shows the results of applying \Cref{thm:small_n_weight} for different sizes $n$ of the data set. The interval width $\Delta$ is chosen to maximize the probability bound.}\label{fig:7}
\end{figure}

\section{Conclusion}\label{sec:conclusion}

In \cite{fluck19}, a precise connection between tangles, a subject from structural graph theory, and clustering, a central technique in data science, was established.
Our work aims at setting the foundation for an in-depth quantitative analysis of tangles in data sets as a means to capture soft cluster assignments.
To this end, we have applied tangle theory to sets of data points that are drawn from multiple Gaussian distributions.
For the two standard graph models in this context, we have found explicit conditions under which, asymptotically almost surely, incomparable tangles that can be associated with the participating Gaussian distributions exist.
Here, we contribute two things.
First we give a probability bound with exponential convergence to $1$, which has strong conditions on the distributions but provides a lower bound on the tangle order.
Secondly we provide a probability bound with slower convergence to $1$, which has weaker assumptions on the data distributions and becomes meaningful for large data sets.

Future projects shall continue the investigation of the potential of tangles in data sets.
For example, it would be interesting to study the \completegraph\ in more detail and possibly with other tools.
More refined bounds on the tangle order would lead to stronger probability bounds.
As another natural follow-up study, we leave as an open problem to improve on our probability bounds (or the conditions on the data distributions) in higher dimensions.
To exploit the tangle potential in the Gaussian setting in more depth, we suggest the theoretical analysis of multiple Gaussian distributions whose means are in arbitrary position or whose standard deviations differ.
In a next step, the results could be applied to actually recover (with some bounded error probability) the hidden labels. 
One can also go beyond Gaussian mixtures and study other ``well-clustered'' data. 

It can also be worthwhile to approach the problem via a study of branch decompositions.
The duality between tangle orders and orders of branch decompositions (\cite{minorsX}, also \cite[Theorem 6.1]{grohe2016}) will yield upper bounds on the probability that tangles of high order exist, which may help in the search for incomparable tangles.

Finally, regarding the benefit of the formalized connection between clustering and tangles, there is not only a potential of tangles to yield theoretical insights such as the presented ones about the existence and quality of clusters in data sets.
Indeed, exploring the connection further, the field of structural graph theory might benefit from the broad research on fast approximation algorithms for clustering, possibly yielding efficient approximations for computing graph tangles or decompositions.
However, note that there may be tangles not stemming from clusters, so clustering algorithms will not necessarily detect \emph{all} tangles in the data.

\bibliographystyle{abbrv}
\bibliography{references}

\end{document}